\documentclass[a4paper,leqno]{amsart}

\usepackage[T1]{fontenc}
\usepackage[utf8]{inputenc}

\usepackage[ngerman,english]{babel}

\usepackage[colorlinks=true, pdfstartview=FitV, linkcolor=blue, %
citecolor=blue, urlcolor=blue]{hyperref}

\usepackage{amssymb,amsthm,mathtools,mathrsfs,extarrows,bbm}
\usepackage[dvipsnames]{xcolor}

\usepackage{tikz}
\usetikzlibrary{cd,fadings}

\usepackage{xparse}

\usepackage[capitalize,nosort]{cleveref}

\usepackage{lscape} 

\theoremstyle{definition}
\newtheorem{defi}{Definition}[section]
\newtheorem{lemmadefi}[defi]{Lemma-Definition}
\newtheorem{notation}[defi]{Notation}
\newtheorem{cond}[defi]{Condition}
\newtheorem*{ack}{Acknowledgements}
\theoremstyle{remark}
\newtheorem*{rem}{Remark}
\theoremstyle{plain}
\newtheorem{lemma}[defi]{Lemma}
\newtheorem{prop}[defi]{Proposition}
\newtheorem{cor}[defi]{Corollary}
\newtheorem{thm}[defi]{Theorem}

\numberwithin{equation}{section} 

\newcommand{\defeq}{\vcentcolon=}
\newcommand{\eqdef}{=\vcentcolon}
\newcommand{\defiff}{\vcentcolon\Longleftrightarrow}
\renewcommand{\setminus}{\smallsetminus}
\renewcommand{\emptyset}{\varnothing}

\newcommand{\D}{\mathcal{D}}
\newcommand{\cO}{\mathcal{O}}
\newcommand{\cM}{\mathcal{M}}
\newcommand{\cN}{\mathcal{N}}
\newcommand{\cE}{\mathcal{E}}
\newcommand{\cR}{\mathcal{R}}
\newcommand{\cF}{\mathcal{F}}
\newcommand{\cG}{\mathcal{G}}
\newcommand{\cA}{\mathcal{A}}
\newcommand{\cS}{\mathcal{S}}
\newcommand{\cH}{\mathcal{H}}

\newcommand{\C}{\mathbb{C}}

\newcommand{\PP}{\mathbb{P}}
\newcommand{\R}{\mathbb{R}}
\newcommand{\Z}{\mathbb{Z}}
\newcommand{\bM}{\mathbb{M}}

\newcommand{\RR}{\mathsf{R}}
\newcommand{\DD}{\mathsf{D}}
\newcommand{\kk}{\mathbf{k}}
\newcommand{\sM}{\mathsf{M}}
\newcommand{\eps}{\varepsilon}

\newcommand{\Db}[1]{\mathrm{D}^\mathrm{b}(#1)}
\newcommand{\Modk}[1]{\mathrm{Mod}(\kk_{#1})}
\newcommand{\Dbk}[1]{\mathrm{D}^\mathrm{b}(\mathbf{k}_{#1})}

\newcommand{\EbIk}[1]{\mathrm{E}^\mathrm{b}(\mathrm{I}\mathbf{k}_{#1})}

\newcommand{\ModD}[1]{\mathrm{Mod}(\D_{#1})}
\newcommand{\DbD}[1]{\mathrm{D}^\mathrm{b}(\D_{#1})}
\newcommand{\ModholD}[1]{\mathrm{Mod}_\mathrm{hol}(\D_{#1})}
\newcommand{\DbholD}[1]{\mathrm{D}^\mathrm{b}_\mathrm{hol}(\D_{#1})}
\newcommand{\op}{\mathrm{op}}
\newcommand{\ModGaussCr}{\mathrm{Mod}_{\textrm{\normalfont Gauß}}^{C,\ranks}(\D_{\PP})}

\newcommand{\EbGaussCr}{\mathrm{E}^{C,\ranks}_{\textrm{\normalfont Gauß}}(\mathrm{I}\mathbf{k}_{\PP})}

\newcommand{\StokesCr}{\mathfrak{SD}^{C,\theta_0,\ranks}}

\newcommand{\SolE}[1]{\mathcal{S}ol_{#1}^\mathrm{E}}
\newcommand{\DRE}[1]{\mathcal{DR}_{#1}^\mathrm{E}}

\newcommand{\ptild}{\widetilde{p}}
\newcommand{\qtild}{\widetilde{q}}

\newcommand{\wf}{\check{w}}
\newcommand{\tf}{\check{t}}

\DeclareFontFamily{U}{bbold}{}
\DeclareFontShape{U}{bbold}{m}{n}
{
	<-5.5> s*[1.069] bbold5
	<5.5-6.5> s*[1.069] bbold6
	<6.5-7.5> s*[1.069] bbold7
	<7.5-8.5> s*[1.069] bbold8
	<8.5-9.5> s*[1.069] bbold9
	<9.5-11> s*[1.069] bbold10
	<11-15> s*[1.069] bbold12
	<15-> s*[1.069] bbold17
}{}
\DeclareRobustCommand{\1}{\text{\usefont{U}{bbold}{m}{n}1}}
\newcommand{\ranks}{\mathchoice{\scalebox{0.84}{\textsc{i}}\hspace{-0.3ex}\mathrm{r}}{\scalebox{0.84}{\textsc{i}}\hspace{-0.3ex}\mathrm{r}}{\scalebox{0.59}{\textsc{i}}\hspace{-0.25ex}\mathrm{r}}{\scalebox{0.47}{\textsc{i}}\hspace{-0.2ex}\mathrm{r}}}

\newcommand{\rec}{c_1} 
\newcommand{\imc}{c_2} 
\newcommand{\red}{d_1}
\newcommand{\imd}{d_2}
\newcommand{\rew}{w_1}
\newcommand{\imw}{w_2}
\newcommand{\rewf}{\wf_1}
\newcommand{\imwf}{\wf_2}

\newcommand{\pc}[1]{c_{(#1)}}

\NewDocumentCommand{\phirp}{g}{\varphi_{\mathrm{r}\IfNoValueTF{#1}{}{,#1}}^+}
\NewDocumentCommand{\phirm}{g}{\varphi_{\mathrm{r}\IfNoValueTF{#1}{}{,#1}}^-}
\NewDocumentCommand{\philp}{g}{\varphi_{\mathrm{l}\IfNoValueTF{#1}{}{,#1}}^+}
\NewDocumentCommand{\philm}{g}{\varphi_{\mathrm{l}\IfNoValueTF{#1}{}{,#1}}^-}
\NewDocumentCommand{\dif}{g}{\eta\IfNoValueTF{#1}{}{_{#1}}}
\newcommand{\psirp}{\psi_{\mathrm{r}}^+}
\newcommand{\psirm}{\psi_{\mathrm{r}}^-}
\newcommand{\psilp}{\psi_{\mathrm{l}}^+}
\newcommand{\psilm}{\psi_{\mathrm{l}}^-}

\newcommand{\real}[1]{\operatorname{Re}#1}

\newcommand{\Ob}{\mathrm{Ob}\,}
\newcommand{\Hc}[1]{\mathrm{H}_\mathrm{c}^{#1}}

\newcommand{\To}{\longrightarrow}
\newcommand{\TO}[1]{\xlongrightarrow{#1}}
\newcommand{\ToPO}{\xlongrightarrow{+1}}
\newcommand{\iso}{\simeq}

\newcommand{\Hom}[1]{\mathrm{Hom}_{#1}}

\newcommand{\Aut}[1]{\mathrm{Aut}_{#1}}

\newcommand{\tensD}{\mathbin{\otimes^{\mathsf{D}}}}
\newcommand{\conv}{\mathbin{\overset{+}{\otimes}}}
\newcommand{\convs}{\overset{*}{\otimes}}
\newcommand{\tensL}[1]{\mathbin{\otimes^{\mathsf{L}}_{#1}}}
\newcommand{\pitens}[1]{\pi^{-1}\kk_{#1}\otimes}

\newcommand{\indlim}[1]{\mathop{``\varinjlim"}\limits_{#1}}
\newcommand{\filim}[1]{\mathop{\varinjlim}\limits_{#1}}

\newcommand{\Sing}[1]{\mathrm{SingSupp}(#1)}

\newcommand{\kE}[1]{\kk_{#1}^\mathrm{E}}

\newcommand{\OmegaE}[1]{\Omega_{#1}^\mathrm{E}}

\newcommand{\ExpD}[3]{\mathcal{E}^{#1}_{#2|#3}}
\NewDocumentCommand{\Exps}{mmg}{\mathsf{E}^{#1}_{#2\IfValueT{#3}{|#3}}}
\NewDocumentCommand{\ExpS}{mmmg}{\mathsf{E}^{#1\rhd #2}_{#3\IfValueT{#4}{|#4}}}
\newcommand{\Expi}[3]{\mathbb{E}^{#1}_{#2|#3}}
\newcommand{\ExpI}[4]{\mathbb{E}^{#1\rhd #2}_{#3|#4}}
\newcommand{\expg}{\bigoplus_{c\in C}\big(\Exps{-\real{\frac{c}{2}z^2}}{\C}{\C}\big)^{r_c}} 
\newcommand{\expgzon}[1]{\bigoplus\limits_{c\in C}\Exps{-\real{\frac{c}{2}z^2}}{#1}} 
\newcommand{\Expg}{\bigoplus_{c\in C}\big(\Expi{-\real{\frac{c}{2}z^2}}{\C}{\PP}\big)^{r_c}} 
\newcommand{\expgwon}[1]{\bigoplus_{c\in C}\Exps{\real{\frac{1}{2c}w^2}}{#1}} 

\newcommand{\fstalk}[2]{{#1}\widehat{|}_{#2}}

\newcommand{\dF}[1]{{}^\mathsf{L}#1}
\newcommand{\sF}[1]{{}^\mathcal{L}#1}  

\newcommand{\vsum}[2]{\begin{matrix} #1 \\ \oplus \\ #2 \end{matrix}}

\newcommand{\katare}{\text{\usefont{U}{min}{m}{n}\symbol{'354}}}
\DeclareFontFamily{U}{min}{}
\DeclareFontShape{U}{min}{m}{n}{<-> udmj30}{}

\title{D-modules of pure Gaussian type and enhanced ind-sheaves}

\author[A.~Hohl]{Andreas Hohl}
\address[A.~Hohl]{Fakultät f\"ur Mathematik, Technische Universit\"at Chemnitz, 09107 Chemnitz, Germany}
\email{andreas.hohl@math.tu-chemnitz.de}
\thanks{The author was partially supported by a doctoral scholarship of \emph{Studienstiftung des deutschen Volkes}.}
\thanks{This is a post-peer-review, pre-copyedit version of an article published in \textsl{manuscripta mathematica}. The final authenticated version is available online at: \href{https://doi.org/10.1007/s00229-021-01281-y}{https://doi.org/10.1007/s00229-021-01281-y}}


\begin{document}

\begin{abstract}
Differential systems of pure Gaussian type are examples of D-modules on the complex projective line with an irregular singularity at infinity, and as such are subject to the Stokes phenomenon. We employ the theory of enhanced ind-sheaves and the Riemann--Hilbert correspondence for holonomic D-modules of A.\ D'Agnolo and M.\ Kashiwara to describe the Stokes phenomenon topologically. Using this description, we perform a topological computation of the Fourier--Laplace transform of a D-module of pure Gaussian type in this framework, recovering and generalizing a result of C.\ Sabbah.
\end{abstract}

\maketitle

\tableofcontents

\section{Introduction}
The study of D-modules with irregular singularities has recently experienced new impulses by a remarkable result of A.\ D'Agnolo and M.\ Kashiwara, the Riemann--Hilbert correspondence for holonomic D-modules (see \cite{DK16}). It states that on a complex manifold $X$ there is a fully faithful functor
$$\SolE{X}\colon \DbholD{X}^\op\hookrightarrow \mathrm{E}^\mathrm{b}(\mathrm{I}\C_{X}),$$
associating to any holonomic D-module an object in the category of \emph{enhanced ind-sheaves} from which one can reconstruct the D-module. The construction of the target category is technical, but it is related to sheaf theory of vector spaces and hence of a topological nature. The theory has since been applied to the study of Stokes phenomena and Fourier--Laplace transforms (see e.g.\ \cite{KS16b}, \cite{DK18}, \cite{DHMS}, \cite{IT18}). Other recent approaches to the study of Fourier transforms of Stokes data have been developed in \cite{Moc10} and \cite{Moc18}.

In their original article \cite[§9.8]{DK16}, the authors give an outlook on a topological study of the Stokes phenomenon of a D-module. In this paper, we develop rigorously these ideas in the case of D-modules of pure Gaussian type $\cM$, meromorphic connections on $\PP=\PP^1(\C)$ with a unique (and irregular) singularity at $\infty$ and exponential factors $-\frac{c}{2z'}$ in the corresponding Levelt--Turrittin decomposition (for $z'$ a local coordinate at $\infty$). In this precise form they were studied by C.\ Sabbah in \cite{Sab16} using Deligne's approach of Stokes-filtered local systems (see \cite{Del}, \cite{Mal91} and \cite{Sab13}) in order to find a transformation rule for the Stokes data attached to such a module. Similar (and more general) systems of differential equations with exponents of pole order $2$ have already been introduced by P.\ Boalch in \cite{Bo12} and \cite{Bo15} (where they are called ``type 3'' connections) with a different motivation. In the latter article, the author shows that a large class of certain quiver varieties arises as moduli spaces (wild character varieties) of such systems and uses this result to construct symplectic isomorphisms between these moduli. The study of Fourier--Laplace transforms is especially interesting in the Gaussian case since this class is invariant: The Fourier--Laplace transform of this kind of system has again a formal type with exponential factors of pole order $2$. Moreover, studying these connections is a natural step further, given that the theory of enhanced ind-sheaves has already proved to be useful in the case of exponents of pole order $1$ (cf.\ e.g.\ \cite{DHMS}), which play a prominent role in mirror symmetry.

It is the main purpose of the present article to reconstruct the results of \cite{Sab16} about the Fourier--Laplace transform of Stokes data with the new methods and to show how these computations can without much effort be adapted to more general cases. This research is based upon the dissertation \cite{Hohl}.

Let us briefly outline the main ideas and the structure of the article:

In the second section, we recall the basic notation and results from the theories of D-modules and enhanced ind-sheaves.

The third section then collects well-known results about Stokes phenomena: Classically, the Stokes phenomenon manifests itself in the fact that a formal solution of a differential equation has different convergent asymptotic lifts in different sectors around an irregular singularity. In the language of D-modules, this is expressed by the statement that the formal Levelt--Turrittin decomposition can locally (on sufficiently small sectors) be lifted to an analytic decomposition. By the Riemann--Hilbert correspondence, this induces a decomposition of the associated topological object $\SolE{X}(\cM)$ (\cref{decompositionSolE}).

In Sections~\ref{chapGaussianType}--\ref{sectionClosedSectorsOfInfiniteRadius}, we introduce the notion of D-modules of pure Gaussian type in the language of D-modules and describe step-by-step the topological object of enhanced solutions $\SolE{\PP}(\cM)$ of such a D-module $\cM$: Starting from the Stokes phenomenon, which yields a direct sum decomposition on small sectors, we discuss how large the radius and angular width of these sectors may be, introducing notions like Stokes multipliers in this framework.
It will finally turn out (\cref{thmBigSectors}) that $\SolE{\PP}(\cM)$ is described by an ordinary sheaf on $\C\times\R$, which in turn is determined by a small set of linear algebra data, the \emph{Stokes data}. In the spirit of \cite{Sab16}, we present Stokes data and a Riemann--Hilbert correspondence for D-modules of pure Gaussian type in Section~\ref{sectionStokesData}.

We will then use this description to compute the Fourier--Laplace transform of a D-module of pure Gaussian type and describe its Stokes data in terms of the Stokes data of the original system. This computation, too, will involve (constructible) sheaves rather than enhanced ind-sheaves in the end and will therefore reduce to calculations in algebraic topology (cohomology groups with compact support). For this purpose, we recall the notions of Fourier--Laplace transform for D-modules and enhanced ind-sheaves in Section~\ref{sectionFourier} before we carry out our computations in the two final sections.

Compared to the approach via Stokes-filtered local systems, our considerations have various advantages: Although the theory is a priori more involved, it turns out that the actual computations to be made are computations in the theory of sheaves of vector spaces and algebraic topology. In particular, one does not need to deal with filtrations, which are often more intricate to handle.
A particularly nice feature of this new approach in dealing with integral transforms is the fact that the functor $\SolE{X}$, which we use for translating between D-modules and topology, is compatible with proper direct images. In the context of Stokes filtrations, the Riemann--Hilbert functor does not have this property. Instead, it was necessary to deal with sequences of blow-ups to compute direct images (cf.\ \cite{HS} and\cite{Sab16}), using a result of Mochizuki (\cite{Moc14}).
Finally, our method of computation needs less input in the following sense: By results like the stationary phase formula (see \cite{Sab08} and \cite{DK18}), we could know a priori that the Fourier--Laplace transform of a D-module of pure Gaussian type is again of pure Gaussian type, and we can explicitly write down the exponential factors of the Fourier--Laplace transform. However, this a-priori-knowledge does not enter our arguments, but is rather obtained as a by-product of our computations automatically.

Our main results are the following: We first recover a theorem of C.\ Sabbah (\cite[Theorem 4.2]{Sab16}), who proved an explicit transformation rule for Stokes data in the case where all the parameters $c$ appearing in the exponential factors share the same argument. In \cref{thmAlignedParametersEnhancedSheaves}, we prove such a transformation rule for enhanced sheaves of pure Gaussian type, which as a corollary (\cref{corAlignedParametersDmodules}) yields the result from loc.\ cit. We then show how such a result can be generalized to situations with weaker assumptions on the parameters. Therefore, we treat a more general case (\cref{thmNotAligned}), illustrating how the methods of the above theorem are naturally adapted to other situations.

\section{Enhanced ind-sheaves and D-modules}
Let $X$ be a complex manifold. We denote the field of complex numbers by $\kk=\C$. We mainly use the notation of \cite{KS90} and \cite{DK16}.

Denote by $\D_X$ the sheaf of rings of differential operators on $X$, by $\ModD{X}$ the category of (left) $\D_X$-modules and by $\DbD{X}$ its bounded derived category. Let $\ModholD{X}$ (resp.\ $\DbholD{X}$) be the full subcategory of objects which are holonomic (resp.\ have holonomic cohomologies). (We refer to \cite{Kas03}, \cite{HTT} and \cite{Bjo} for details on D-modules.)

In \cite{DK16}, the authors introduced the triangulated category $\EbIk{X}$ of enhanced ind-sheaves on $X$ as a quotient of the derived category of ind-sheaves on $X\times(\R\sqcup\{\pm\infty\})$. Together with the convolution product $\conv$ it is a tensor category, and an important object is $$\kE{X}\defeq \indlim{a\to \infty}\kk_{\{t\geq a\}}.$$ They proved the following result, which is a generalization of the classical Riemann--Hilbert correspondence (see \cite{Kas84}) to not necessarily regular holonomic D-modules.

\begin{thm}[cf.\ {\cite[Theorem 9.5.3]{DK16}}]\label{thmRH}
The functor of enhanced solutions
$$\SolE{X}\colon \DbholD{X}^\op\to \EbIk{X}$$
is fully faithful.
\end{thm}
By this result, the object $\SolE{X}(\cM)$ is the topological counterpart of a D-module $\cM$, containing all the information about $\cM$. In particular, it must encode the Stokes phenomenon.

We refer to \cite{DK16} for further details on enhanced ind-sheaves (see also \cite{DK19}, \cite{KS16a}, \cite{Kas16}). Let us only recall the bifunctor (cf.\ \cite[Definition 4.5.2]{DK16})
$$\pi^{-1}(\bullet)\otimes (\bullet)\colon \Dbk{X}\times\EbIk{X}\to\EbIk{X},$$
where $\pi\colon X\times\R\to X$ is the projection. This functor enables us to consider the ``restriction'' $\pi^{-1}\kk_Z\otimes K$ of an enhanced ind-sheaf $K\in\EbIk{X}$ to a locally closed subset $Z\subseteq X$. (One considers this object rather than the inverse image along the embedding since it keeps track of the behaviour at the boundary of $Z$.) In this way, one can use gluing techniques similar to sheaf theory in $\EbIk{X}$ by carrying over sequences of sheaves like
$$0\To \kk_{Z_1\cup Z_2}\To \kk_{Z_1}\oplus \kk_{Z_2}\To \kk_{Z_1\cap Z_2}\To 0$$
for two closed subsets $Z_1,Z_2\subseteq X$. Thus, given a description of an enhanced ind-sheaf on two sets, one can obtain a description on their union. In fact, although the third object of a distinguished triangle is generally unique up to (non-unique) isomorphism only, uniqueness will always be guaranteed in our constructions (by \cite[Proposition 10.1.17]{KS06} or \cite[Corollary IV.1.5]{GM03}).

\subsection{Enhanced sheaves}
There is a natural functor $\Dbk{X\times\R}\to \EbIk{X}$, and we consider sheaves on $X\times\R$ as enhanced ind-sheaves through this functor. Objects of $\Dbk{X\times\R}$ will be called \emph{enhanced sheaves on $X$}. (Note that other authors usually define the category of enhanced sheaves as a certain subcategory of $\Dbk{X\times\R}$, cf.\ \cite{DK18}, \cite{DHMS}. We will not introduce it here, although we actually work in this subcategory.)

There is a \emph{convolution product} on $\Dbk{X\times\R}$ defined by
$$\cF\convs \cG\defeq \RR\mu_!(q_1^{-1}\cF\otimes q_2^{-1}\cG),$$
where the maps $\mu,q_1,q_2\colon X\times\R^2\to X\times\R$ are given by $\mu(x,t_1,t_2)=(x,t_1+t_2)$, $q_1(x,t_1,t_2)=(x,t_1)$ and $q_2(x,t_1,t_2)=(x,t_2)$. Via the natural functor above it corresponds to the convolution functor $\conv$ for enhanced ind-sheaves.

For a locally closed subset $Z\subseteq X$, we will write $\cF_Z\defeq \cF_{Z\times \R}$.

\subsection{Exponential enhanced (ind-)sheaves} \label{sectionEnhancedExponentials} We recall here the definition of enhanced exponentials as introduced in \cite{DK18}.

Let $U\subseteq X$ be an open subset and let $\varphi,\varphi^-,\varphi^+\colon U\to \R$ be continuous functions. Let moreover $Z\subseteq U$ be locally closed with $\varphi^-(x)\leq \varphi^+(x)$ for any $x\in Z$. 

We consider the enhanced sheaves
\begin{align*}
\Exps{\varphi}{Z}{X}&\defeq \pi^{-1}\kk_Z\otimes \kk_{\{t+\varphi\geq 0\}}\in \Dbk{X\times\R},\\
\ExpS{\varphi^+}{\varphi^-}{Z}{X}&\defeq \pi^{-1}\kk_Z\otimes\kk_{\{ -\varphi^+\leq t < -\varphi^- \}}\in \Dbk{X\times\R},
\end{align*}
where we write for short $\{ t+\varphi\geq 0 \}\defeq \{ (x,t)\in X\times\R \mid x\in U, t+\varphi(x)\geq 0 \}$, and similarly $\{ -\varphi^+\leq t < -\varphi^- \}\defeq \{ (x,t)\in X\times\R \mid x\in U, -\varphi^+(x)\leq t < -\varphi^-(x) \}$.

Furthermore, we consider the enhanced ind-sheaves
\begin{align*}
\Expi{\varphi}{Z}{X}&\defeq \kE{X}\conv \Exps{\varphi}{Z}{X}\in \EbIk{X},\\
\ExpI{\varphi^+}{\varphi^-}{Z}{X}&\defeq \kE{X}\conv\ExpS{\varphi^+}{\varphi^-}{Z}{X}\in \EbIk{X}.
\end{align*}
The following lemma is an easy observation.
\begin{lemma}\label{boundedExponent}
If $U\subseteq X$ is open, $\varphi,\psi\colon U\to \R$ are continuous functions, and $Z\subseteq U$ is locally closed such that $\varphi-\psi$ is bounded on $Z$, then there is a canonical isomorphism
$$\Expi{\varphi}{Z}{X}\iso \Expi{\psi}{Z}{X}.$$
\end{lemma}
It is a fundamental observation (cf.\ \cite[Corollary 9.4.12]{DK16}) that $\SolE{X}(\ExpD{\varphi}{U}{X})\iso \Expi{\real{\varphi}}{U}{X}$, where $\ExpD{\varphi}{U}{X}$ is the exponential D-module for some meromorphic function $\varphi\in\cO_X(*D)$ with poles on a closed hypersurface $D\subset X$ and $U=X\setminus D$.

\section{Stokes phenomena for enhanced solutions}
In dimension one, the Stokes phenomenon describes the fact that around an irregular singularity formal solutions are not necessarily convergent, but admit asymptotic expansions on sufficiently small sectors. In the language of D-modules, this is known as the theorem of Hukuhara--Turrittin, stating that the formal Levelt--Turrittin decomposition lifts to a local analytic decomposition on the real blow-up space (cf.\ \cite{Mal91}). We now explain how this is expressed in terms of enhanced ind-sheaves.

Let $X=\C$ and let $\cM\in\ModholD{X}$ be a meromorphic connection with pole at $0$, i.e.\ $\cM(*0)\iso\cM$ and $\Sing{\cM}=\{0\}$. Assume $\cM$ has an (unramified) Levelt--Turrittin decomposition at $0$, i.e.\
$$\fstalk{\cM}{0}\iso \bigoplus_{i\in I} \fstalk{(\ExpD{\varphi_i}{U}{X}\tensD \mathcal{R}_i)}{0}$$
for some finite index set $I$, meromorphic functions $\varphi_i\in \cO_X(*0)$ and regular holonomic $\D_X$-modules $\cR_i$. Here, $\fstalk{\cM}{0}\defeq \widehat{\cO_{X,0}}\otimes_{\cO_{X,0}} \cM_0$ is the formal completion of the stalk.

The following result has been stated in \cite[§9.8]{DK16}. It is also given as a corollary of a more general result in \cite[Corollary 3.7]{IT18}. We give a direct proof in the unramified case.
\begin{prop}\label{decompositionSolE}
If $\cM$ has a Levelt--Turrittin decomposition at $0$, then for any direction $\theta\in \R/2\pi\Z$ there exist constants $\eps, R\in \R_{>0}$, determining an open sector $S_\theta=\{ z\in X \mid 0<|z|<R,\; \arg z \in (\theta-\eps,\theta+\eps)\}$, such that we have an isomorphism in $\EbIk{X}$
$$\pi^{-1}\kk_{S_\theta}\otimes \SolE{X}(\cM)\iso \pi^{-1}\kk_{S_\theta}\otimes \bigoplus_{i\in I}(\Expi{\real{\varphi_i}}{U}{X})^{r_i}.$$
\end{prop}

We first establish the following lemma, which is the crucial step in proving the proposition. We denote by $\varpi\colon \widetilde{X}\to X$ the real blow-up of $X$ at $0$ and refer to \cite[§7.3 and §9.2]{DK16} for details and notation regarding D-modules and enhanced De Rham functors on blow-up spaces.
\begin{lemma}\label{lemmaDREBlowUp}
	Let $V\subseteq \widetilde{X}$ be open and $\cN\in\Db{\D_{\widetilde{X}}^\cA}$. Then there is an isomorphism in $\EbIk{\widetilde{X}}$
	$$\pi^{-1}\kk_V\otimes \DRE{\widetilde{X}}(\cN)\iso \pi^{-1}\kk_V\otimes \DRE{\widetilde{X}}(\cN\otimes \kk_{\overline{V}}).$$
\end{lemma}
\begin{proof}
	We will use the notation from \cite{KS01} here to emphasize the difference between the two functors $\iota_{\widetilde{X}}$ and $\beta_{\widetilde{X}}$ from sheaves to ind-sheaves. (Note that $\beta_{\widetilde{X}}$ is often suppressed in the notational conventions of \cite{DK16}.)
	
	By \cite[Lemma 3.3.3]{KS01} and \cite[Proposition 4.2.14]{KS01}, respectively, we have
	$$\iota_{\widetilde{X}}\mspace{1mu}\kk_V \iso \indlim{U\subset\subset \widetilde{X}} \kk_{V\cap U}
	\qquad\text{and}\qquad
	\beta_{\widetilde{X}}\mspace{1mu}\kk_{\overline{V}}\iso \indlim{U\subset\subset \widetilde{X},W\supset \overline{V}}\kk_{U\cap \overline{W}}$$
	and therefore
		$$\beta_{\widetilde{X}}\mspace{1mu}\kk_{\overline{V}}\otimes \iota_{\widetilde{X}}\mspace{1mu}\kk_V \iso \indlim{U\subset\subset \widetilde{X}, U'\subset\subset\widetilde{X}}\kk_{V\cap U\cap U'}\iso \iota_{\widetilde{X}}\mspace{1mu}\kk_V,$$
	since $U\cap U'$ ranges through the family of all relatively compact open subsets of $\widetilde{X}$ as $U$ and $U'$ do.
	This now enables us to use \cite[Theorem 5.4.19]{KS01} and obtain
	\begin{align*}
		\pi^{-1}\kk_V\otimes \DRE{\widetilde{X}}(\cN) &\iso (\OmegaE{\widetilde{X}}\tensL{\beta\mspace{1mu}\pi^{\mspace{1mu}-1}\D_{\widetilde{X}}^\cA} \beta\mspace{1mu}\pi^{\mspace{1mu}-1}\cN) \otimes \iota_{\widetilde{X}\times \overline{\R}}\mspace{2mu}\pi^{\mspace{1mu}-1}\kk_V\\
		&\iso (\OmegaE{\widetilde{X}}\tensL{\beta\mspace{1mu}\pi^{\mspace{1mu}-1}\D_{\widetilde{X}}^\cA} \beta\mspace{1mu}\pi^{\mspace{1mu}-1}\cN)\otimes (\beta\mspace{1mu}\pi^{\mspace{1mu}-1}\kk_{\overline{V}}\otimes \iota_{\widetilde{X}\times \overline{\R}}\mspace{2mu}\pi^{\mspace{1mu}-1}\kk_V)\\
		&\iso \big(\OmegaE{\widetilde{X}}\tensL{\beta\mspace{1mu}\pi^{\mspace{1mu}-1}\D_{\widetilde{X}}^\cA} \beta\mspace{1mu}\pi^{\mspace{1mu}-1}(\cN\otimes \kk_{\overline{V}})\big)\otimes \iota_{\widetilde{X}\times \overline{\R}}\mspace{2mu}\pi^{\mspace{1mu}-1}\kk_V\\
		&\iso \pi^{-1}\kk_V\otimes \DRE{\widetilde{X}}(\cN\otimes \kk_{\overline{V}}).
	\end{align*}\vspace{-1cm}

\end{proof}\vspace{0.2cm}

\begin{proof}[Proof of \cref{decompositionSolE}]
Still using the notation of \cite[§7.3 and §9.2]{DK16} and in particular \cite[Corollary 9.2.3]{DK16}, one can compute
\begin{align*}
\pi^{-1}\kk_{S_\theta} \otimes \DRE{X}(\cM) 
&\iso \mathrm{E}\varpi_{!!}\big( \pi^{-1}\kk_{\varpi^{-1}(S_\theta)} \otimes \DRE{\widetilde{X}}(\cM^\cA) \big)\\
&\overset{(\star)}{\iso} \mathrm{E}\varpi_{!!} \big( \pi^{-1}\kk_{\varpi^{-1}(S_\theta)} \otimes \DRE{\widetilde{X}}(\cM^\cA \otimes \kk_{\overline{\varpi^{-1}(S_\theta)}}) \big)\\
&\overset{(\blacktriangle)}{\iso} \mathrm{E}\varpi_{!!} \Big( \pi^{-1}\kk_{\varpi^{-1}(S_\theta)} \otimes \DRE{\widetilde{X}}\big(\bigoplus_{i\in I}\big( (\ExpD{\varphi_i}{U}{X})^\cA \big)^{r_i}\otimes \kk_{\overline{\varpi^{-1}(S_\theta)}}\big) \Big)\\
&\iso \pi^{-1}\kk_{S_\theta}\otimes \bigoplus_{i\in I} \big(\DRE{X}(\ExpD{\varphi_i}{U}{X})\big)^{r_i}.
\end{align*}
Here $(\star)$ follows from \cref{lemmaDREBlowUp} and $(\blacktriangle)$ follows from the classical Hukuhara--Turrittin theorem if $\eps$ and $R$ are small enough.

The statement about the solution functor (instead of the De Rham functor) is easily deduced by duality.
\end{proof}

The Stokes phenomenon arises through the fact that the isomorphism from \cref{decompositionSolE} depends on $\theta$.
Note that, in contrast to common formulations of the Hukuhara--Turrittin theorem (cf.\ e.g.\ \cite[Théorème (1.4)]{Mal91}), the statement of \cref{decompositionSolE} does not involve blow-ups.

\section{D-modules of pure Gaussian type}\label{chapGaussianType}
Let $\PP=\PP^1(\C)$ be the analytic complex projective line, denote by $\C=\PP\smallsetminus\{\infty\}$ the affine chart with local coordinate $z$ at $0$ and by $j\colon \C\hookrightarrow \PP$ its embedding.

\begin{defi}\label{defGaussian}
Let $C\subset\C^\ast$ be a finite (non-empty) subset. A holonomic $\D_\PP$-module $\cM$ is said to be of \emph{pure Gaussian type $C$} if the following conditions hold:
\begin{itemize}
\item[(a)] $\cM\simeq \cM(*\infty)$ (as $\D_\PP$-modules).
\item[(b)] $\Sing{\cM}=\{\infty\}$.
\item[(c)] There exist regular holonomic $\D_\PP$-modules $\cR_c$ such that $\cM$ has a Levelt--Turrittin decomposition at $\infty$ of the form
$$\cM\widehat{|}_\infty\iso \bigoplus_{c\in C} \left( \cE^{-\frac{c}{2}z^2}_{\C|\PP}\tensD\cR_c \right)\widehat{|}_\infty.$$
\end{itemize}
In other words, $\cM$ is a meromorphic connection on $\PP$ with a pole at $\infty$ and an (unramified) Levelt--Turrittin decomposition at $\infty$ with exponential factors $-\frac{c}{2}z^2$. (Note that polynomial functions in $z$ extend to meromorphic functions on $\PP$.)
	
The rank of $\cR_c$ will be denoted by $r_c$ and the family of these ranks will be denoted by $\ranks\defeq (r_c)_{c\in C}$.
\end{defi}

Some properties of the enhanced solutions of such D-modules of pure Gaussian type are collected in the following lemma.

\begin{lemma}\label{GaussBasicProperties}
Let $\cM\in \ModholD{\PP}$ be of pure Gaussian type $C$. Then:
\begin{itemize}
\item[(i)] $\pi^{-1}\kk_{\C}\otimes\SolE{\PP}(\cM)\iso \SolE{\PP}(\cM)$.
\item[(ii)] For any direction $\theta\in\R/2\pi\Z$, there exists a small sector $S_\theta$ at $\infty$ (with central direction $\theta$) such that $$\pi^{-1}\kk_{S_\theta}\otimes\SolE{\PP}(\cM)\iso \pi^{-1}\kk_{S_\theta}\otimes \bigoplus\limits_{c\in C}\big(\Expi{-\real{\frac{c}{2}z^2}}{\C}{\PP}\big)^{r_c}.$$
\item[(iii)] For any open $B\subset \C$ such that $\overline{B}\subset \C$ (where $\overline{B}$ denotes the closure of $B$ in $\PP$), one has
$$\pi^{-1}\kk_B \otimes \SolE{\PP}(\cM)\iso \pi^{-1}\kk_B\otimes (\kE{\PP})^r.$$
\end{itemize}
\end{lemma}
\begin{proof}
The statements (i) and (ii) directly follow from Proposition \cite[Corollary 9.4.11]{DK16} and Proposition \ref{decompositionSolE}, respectively.

The third assertion is proved using \cite[Lemma 2.7.6]{DK19} and the fact that $\cM$ is non-singular outside $\infty$.
\end{proof}
On the other hand, since $\PP$ is compact, we have the following statement about the global structure of $\SolE{X}(\cM)$. It is a direct application of \cite[Lemma 2.5.1]{DHMS} (cf.\ also \cite[Definition 4.9.2 and Theorem 9.3.2]{DK16}).
\begin{lemma}\label{propGlobalEnhSheaf}
Let $\cM\in\ModholD{\PP}$ be of pure Gaussian type. Denote by $\tilde{j}\colon \C\times\R\hookrightarrow\PP\times\R$ the embedding. There exists $\cF\in\Dbk{\C\times \R}$ such that
$$\SolE{\PP}(\cM)\iso \kE{\PP}\conv \tilde{j}_!\cF.$$
\end{lemma}
Thus, the enhanced solutions of a D-module of pure Gaussian type are determined by a globally defined enhanced sheaf which restricts to zero on the singularity. The aim of the next sections will be to describe such an enhanced sheaf, and this goal is achieved in \cref{thmBigSectors}.

\section{Stokes directions and width of sectors}

Let $C\subset \C^\times$ be a finite subset and let $\cM\in \ModholD{\PP}$ be of pure Gaussian type $C$.

In this section, we extend the decomposition from \cref{GaussBasicProperties} (ii) to a decomposition of $\SolE{\PP}(\cM)$ on sectors around $\infty$ that intersect at most one Stokes line for each pair $c,d\in C$. That is, we give a more precise description of how ``small'' the sectors' width has to be.

As we have seen, the enhanced solutions of $\cM$ are not interesting at the singularity but in close neighbourhoods, which are then subsets of $\C=\PP\setminus\{\infty\}$. Therefore, we will set up everything in the complex plane.

\begin{lemmadefi}\label{defStokesLines}
Let $c,d\in C$, $c\neq d$. The set
$$\mathrm{St}_{c,d}=\left\{z\in \C \left| -\real{\frac{c}{2}z^2}=-\real{\frac{d}{2}z^2}\right. \right\}$$
is the union of four closed half-lines with initial point $0$, perpendicular to one another.
These half-lines are called the \emph{Stokes lines} of the pair $c,d$. Their directions (i.e.\ the arguments of the points on the Stokes lines) are called the \emph{Stokes directions} (of the pair $c,d$).

We say that a direction is \emph{generic} if it is not a Stokes direction for any pair $c,d\in C$.
\end{lemmadefi}

\begin{defi}\label{defSectorAtInfty}
A subset $S\subset \PP$ is said to be
\begin{itemize}
\item[$\blacktriangleright$] an \emph{open sector} at $\infty$ if
$$S=\{ z\in \C \mid R<|z|<\infty, \arg z \in (\theta-\eps,\theta+\eps) \}\subseteq \C\subset \PP$$
for some $R\in \R_{\geq 0}$, $\eps\in\R_{>0}$ and $\theta\in\R/2\pi\Z$.
\item[$\blacktriangleright$] a \emph{closed sector} at $\infty$ if
$$S=\{ z\in \C \mid R\leq|z|<\infty, \arg z \in [\theta-\eps,\theta+\eps] \text{ for $|z|\neq 0$}\}\subseteq \C\subset \PP$$
for some $R,\eps\in \R_{\geq 0}$ and $\theta\in\R/2\pi\Z$.\\
(For $\eps=0$, this includes the case of half-lines.)
\end{itemize}
The \emph{radius} of such a sector is the number $\frac{1}{R}\in (0,+\infty]$, and its \emph{width} is the number $\min(2\eps,2\pi)\in[0,2\pi]$. Note that a closed sector at $\infty$ is topologically closed in $\C$ (but not in $\PP$).

We will say that an (open or closed) sector at $\infty$ \emph{contains} a direction $\theta\in\R/2\pi\Z$ if its intersection with the open half-line $\{ z\in\C\setminus \{0\} \mid \arg z = \theta \}$ is non-empty.
\end{defi}
\noindent On sectors containing no Stokes direction, we can introduce an order on $C$.
\begin{notation}
Let $S$ be a sector at $\infty$ and let $c,d\in C$. We write
$$c<_S d \quad\defiff \quad \real{\frac{c}{2}z^2} < \real{\frac{d}{2}z^2} \textnormal{ for all $z\in S\setminus \{0\}$}.$$
For $\theta\in \R/2\pi\Z$, we write
$$c<_\theta d \quad \defiff\quad \real{\frac{c}{2}z^2} < \real{\frac{d}{2}z^2} \textnormal{ for all $z\in \C$ with $z\neq 0$ and $\arg z=\theta$}.$$
\end{notation}

We now describe morphisms between the exponential enhanced ind-sheaves in the local decomposition of $\SolE{\PP}(\cM)$ from \cref{GaussBasicProperties} (ii). This lemma is analogous to (and inspired by) \cite[Lemma 9.8.1]{DK16}, \cite[Lemma 5.1.2]{DHMS} and \cite[Lemma 4.3.1]{DK18}.

\begin{lemma}\label{lemmaHomExp}
Consider the meromorphic functions $\varphi_1,\varphi_2$ on $\PP$ given by $\varphi_1(z)=-\frac{c}{2}z^2$ and $\varphi_2(z)=-\frac{d}{2}z^2$ for $c,d\in\C^\times$, $c\neq d$. Let $S\subset \PP$ be a sector (open or closed) at $\infty$. Then we have
\begin{align*}
\Hom{\EbIk{\PP}}\big(\Expi{\real{\varphi_1}}{S}{\PP},\Expi{\real{\varphi_2}}{S}{\PP}\big)&\iso \Hom{\Dbk{\C\times \R}}\big(\Exps{\real{\varphi_1}}{S}{\C},\Exps{\real{\varphi_2}}{S}{\C}\big)\\ 
&\iso \begin{cases} \kk & \textnormal{if $\real{\varphi_1} \geq \real{\varphi_2}$ on $S$} \\ 0 & \textnormal{otherwise}\end{cases}.
\end{align*}
Here, the first isomorphism (from right to left) is induced by the functor $\kE{\PP}\conv \tilde{j}_!(\bullet)$ and the second isomorphism is the natural identification of a morphism with multiplication by a complex number.
\end{lemma}
\begin{proof}
Using \cite[Proposition 4.7.9, Lemma 4.4.6 and Corollary 3.2.10]{DK16}, we get
\begin{align*}
&\Hom{\EbIk{\PP}}\big(\Expi{\real{\varphi_1}}{S}{\PP},\Expi{\real{\varphi_2}}{S}{\PP}\big)\\
&\hspace{1cm}\iso \filim{a\geq 0} \Hom{\Dbk{\C\times\R}}\big(\pi^{-1}\kk_S\otimes \kk_{\{ t+\real{\varphi_1}\geq 0 \}},\pi^{-1}\kk_S\otimes \kk_{\{t+\real{\varphi_2}\geq a\}}\big).
\end{align*}
If $\real{\varphi_1}\geq \real{\varphi_2}$ at each point of $S$, these $\mathrm{Hom}$-spaces are isomorphic to $\kk$ for any $a\in\R_{\geq 0}$ (and hence also their direct limit).
If there are points in $S$ where $\real{\varphi_1}<\real{\varphi_2}$, it is not difficult to see that $\real{(\varphi_2-\varphi_1)}$ is not bounded from above on $S$. It follows that the $\mathrm{Hom}$-space is trivial for any $a\in\R_{\geq 0}$ (and hence also the direct limit).
\end{proof}

The following result shows how automorphisms of the Gaussian model on sectors can be interpreted as block matrices.
\begin{prop}\label{autosAsMatrices}
Let $S\subset \PP$ be a sector at $\infty$ and assume that $S$ is not a half-line whose direction is a Stokes direction for some $c,d\in C$.
If we choose a numbering of the elements of $C$, i.e.\ $C=\{ \pc{1},\ldots, \pc{n} \}$, we have
\begin{align*}
&\Aut{\EbIk{\PP}}\Big( \pi^{-1}\kk_S \otimes  \bigoplus_{c\in C} \big(\Expi{-\real{\frac{c}{2}z^2}}{\C}{\PP}\big)^{r_c}  \Big)\\ &\iso \Aut{\Dbk{\PP\times\R}}\Big( \pi^{-1}\kk_S \otimes  \bigoplus_{c\in C} \big(\Exps{-\real{\frac{c}{2}z^2}}{\C}{\PP}\big)^{r_c}  \Big) \\ &\iso \big\{ A=(A_{jk})_{j,k=1}^n\in\kk^{r\times r}\,\big| \, A_{jk}\in \kk^{r_{\pc{j}}\times r_{\pc{k}}}, A_{jj} \textnormal{ is invertible for any } j,\\
&\hspace{4.4cm}A_{jk}=0 \textnormal{ if  $\real{\frac{\pc{j}}{2}z^2}<\real{\frac{\pc{k}}{2}z^2}$ for some $z\in S$} \big\}.
\end{align*}
In particular, if $\pc{1}<_S\pc{2}<\ldots <_S \pc{n}$, then the right hand side consists precisely of the invertible, lower block-triangular matrices with block sizes given by the numbers $r_{\pc{j}}$.
\end{prop}

\begin{prop}\label{propWideSectors}
Let $\cM$ be of pure Gaussian type $C$. For any (open or closed) sector $S$ at $\infty$ of sufficiently small radius intersecting at most one Stokes line for each pair $c,d\in C$, there is an isomorphism
\begin{equation*}
\pi^{-1}\kk_{S}\otimes\SolE{\PP}(\cM)\iso \pi^{-1}\kk_{S}\otimes \bigoplus\limits_{c\in C}\big(\Expi{-\real{\frac{c}{2}z^2}}{\C}{\PP}\big)^{r_c}.
\end{equation*}
\end{prop}
\begin{proof}	
Let us write for short $H\defeq \SolE{\PP}(\cM)$ and $\bM\defeq \bigoplus_{c\in C}\big(\Expi{-\real{\frac{c}{2}z^2}}{\C}{\PP}\big)^{r_c}$.

The following argument enables us to recursively obtain the desired isomorphism by gluing those on small sectors (cf.\ \cref{GaussBasicProperties} (ii)):
Assume that we are given two open sectors $S_1,S_2\subset S$ at $\infty$ with isomorphisms
\begin{align*}
\alpha_j\colon \pi^{-1}\kk_{S_j}\otimes H\TO{\sim} \pi^{-1}\kk_{S_j}\otimes \bM
\end{align*}
for $j\in\{1,2\}$ and assume moreover that $S_1\cap S_2\neq \emptyset$, that we have $S_1\nsubseteq S_2$ and $S_2\nsubseteq S_1$, that $S_2$ contains at most one Stokes direction and no Stokes direction for the same pair $c,d$ is contained in $S_1$.
	
Choose a numbering of the elements of $C$ such that $\pc{1}<_{S_1\cap S_2} \pc{2} <_{S_1\cap S_2} \ldots <_{S_1\cap S_2} \pc{n}$. The isomorphisms $\alpha_j$ induce two isomorphisms
\begin{align*}
\widetilde\alpha_j\colon \pi^{-1}\kk_{S_1\cap S_2}\otimes H\TO{\sim} \pi^{-1}\kk_{S_1\cap S_2}\otimes \bM.
\end{align*}
By Proposition \ref{autosAsMatrices}, the transition isomorphism $\widetilde\alpha_2\circ\widetilde\alpha_1^{-1}$ can be represented by a lower block-triangular matrix $A=(A_{jk})$. One can decompose $A=A''A'$ as follows:

If $S_2$ contains a Stokes direction for the pair $\pc{l},\pc{l'}$ ($l<l'$), let $A'$ be the block matrix (with the same block structure as $A$) having identity matrices on the diagonal and $A'_{l'l}=A_{l'l'}^{-1}A_{l'l}$. All the other blocks of $A'$ are zero.
If $S_2$ contains no Stokes direction, let $A'\defeq \1$. Set $A''\defeq AA'^{-1}$.

It is not difficult to see that, in either of the two cases, the matrix $A'$ represents an automorphism of $\pi^{-1}\kk_{S_1}\otimes \bM$ and the matrix $A''$ represents an automorphism of $\pi^{-1}\kk_{S_2}\otimes \bM$ (by the correspondence of Proposition \ref{autosAsMatrices}).
	
Consider the diagram
\begin{center}
\begin{tikzcd}
\pi^{-1}\kk_{S_1\cap S_2}\otimes H \arrow{r} \arrow{d}{A'\circ \widetilde\alpha_1=A''^{-1}\circ \widetilde\alpha_2} &[-5pt] \pi^{-1}\kk_{S_1}\otimes H \oplus \pi^{-1}\kk_{S_2}\otimes H \arrow{r} \arrow{d}{A''^{-1}\circ \alpha_2}[swap]{A'\circ \alpha_1} &[-5pt] \pi^{-1}\kk_{S_1\cup S_2}\otimes H\arrow{r}{+1} \arrow[dashed]{d}{\widehat{\alpha}}&[-5pt]{}\\
\pi^{-1}\kk_{S_1\cap S_2}\otimes \bM \arrow{r} & \pi^{-1}\kk_{S_1}\otimes \bM \oplus \pi^{-1}\kk_{S_2}\otimes \bM \arrow{r} & \pi^{-1}\kk_{S_1\cup S_2}\otimes \bM\arrow{r}{+1}&{}
\end{tikzcd}
\end{center}
where the rows are distinguished triangles.	By our construction of $A'$ and $A''$, the square on the left of the diagram commutes and the vertical arrows are isomorphisms. Therefore, there exists an isomorphism $\widehat{\alpha}$ completing the diagram to an isomorphism of distinguished triangles.
\end{proof}

\section{Stokes multipliers and monodromy}\label{sectionStokesMultipliers}

Let $C\subset\C^\times$ be a finite subset and $\cM\in\ModholD{\PP}$ be of pure Gaussian type $C$. As we have seen, we generally need four sectors to cover a neighbourhood of $\infty$ by sectors on which we have isomorphisms as in \cref{propWideSectors}.

Fix a generic direction $\theta_0$ and choose a numbering of the elements of $C$ such that $\pc{1}<_{\theta_0} \pc{2}<_{\theta_0}\ldots<_{\theta_0} \pc{n}$. Clearly, $\theta_0+k\frac{\pi}{2}$ (for $k\in\{1,2,3\}$) are also generic. By Proposition \ref{propWideSectors}, there exists $R\in\R_{>0}$ such that on the closed sectors $\Sigma_k\defeq\{ z\in\C \mid |z|\geq R, \arg z\in [\theta_0+(k-1)\frac{\pi}{2},\theta_0+k\frac{\pi}{2}]\}$, $k\in\Z/4\Z$, we have isomorphisms
\begin{equation}\label{eq:AutosOnFourSectors}
\alpha_k\colon \pi^{-1}\kk_{\Sigma_k}\otimes\SolE{\PP}(\cM) \TO{\sim} \pi^{-1}\kk_{\Sigma_k}\otimes \bigoplus\limits_{c\in C}\big(\Expi{-\real{\frac{c}{2}z^2}}{\C}{\PP}\big)^{r_c}.
\end{equation}
(Note that these isomorphisms are not unique, so this step involves a choice.)

On the half-line $\Sigma_{k,k+1}\defeq \Sigma_k\cap \Sigma_{k+1}$, $\alpha_k$ and $\alpha_{k+1}$ induce isomorphisms (by abuse of notation, we denote them by the same symbols)
$$\alpha_k,\alpha_{k+1}\colon \pi^{-1}\kk_{\Sigma_{k-1,k}}\otimes\SolE{\PP}(\cM) \TO{\sim} \pi^{-1}\kk_{\Sigma_{k-1,k}}\otimes \bigoplus\limits_{c\in C}\big(\Expi{-\real{\frac{c}{2}z^2}}{\C}{\PP}\big)^{r_c}$$
and the transition isomorphism $$\alpha_{k+1}\circ (\alpha_k)^{-1}\in\Aut{\EbIk{\PP}}\Big(\pitens{\Sigma_{k,k+1}}\Expg\Big)$$
is represented by an invertible, block-triangular matrix $\sigma_k$ (cf.\ \cref{autosAsMatrices}).

\begin{defi}
The matrices $\sigma_k$ are called \emph{Stokes multipliers} (or \emph{Stokes matrices}) of $\cM$.
\end{defi}
\noindent (Remember that these notions require fixing a generic direction.)

\begin{prop}\label{propTrivialMonodromy}
The (counterclockwise) product of the Stokes multipliers for a D-module of pure Gaussian type is the identity, i.e.\ $\sigma_4\sigma_3\sigma_2\sigma_1= \1$.
\end{prop}
\begin{proof}
Choose $\rho>R$ and set $B\defeq\{ z\in\C \mid |z|\leq \rho \}$. There is a canonical isomorphism (see \cref{boundedExponent})
$$\tau\colon \pi^{-1}\kk_B\otimes \bigoplus\limits_{c\in C}\big(\Expi{-\real{\frac{c}{2}z^2}}{\C}{\PP}\big)^{r_c}\TO{\sim}\pi^{-1}\kk_B\otimes (\kE{\PP})^r.$$
We set $D_j\defeq \Sigma_j\cap B$, $D_{jk}\defeq D_j\cap D_k$ and $D\defeq \bigcup_{j\in\Z/4\Z} D_j$. Moreover, we write for short $H\defeq \SolE{\PP}(\cM)$ and $\bM\defeq \Expg$.

For each $k\in\Z/4\Z$, one has a chain of isomorphisms
$$\pitens{D_k}H\TO{\tau\circ \alpha_k}\pitens{D_k}(\kE{\PP})^r\iso \pi^{-1}(\kk_{D_k})^r \otimes \kE{\PP}.$$
The transition isomorphism on $D_{k,k+1}$ is given by the Stokes multiplier $\sigma_k$ (which can be viewed as an automorphism of the sheaf $(\kk_{D_k})^r$).

Therefore, $\pitens{D}H\iso \pi^{-1}\cG \otimes \kE{\PP}$, where $\cG$ is a local system of rank $r$ on $D$ (extended by zero to $\C$) with monodromy given by $\sigma_4\sigma_3\sigma_2\sigma_1$.

On the other hand, by Lemma \ref{GaussBasicProperties}, we have an isomorphism $\pitens{D}H\iso \pi^{-1}(\kk_D)^r\otimes \kE{\PP}$. Since the functor $\pi^{-1}(\bullet)\otimes \kE{\PP}$ is fully faithful (see \cite[Proposition 4.7.15]{DK16}), $\cG$ is isomorphic to the constant local system and hence that their monodromies are equal.
\end{proof}

\section{A sheaf describing enhanced solutions}\label{sectionClosedSectorsOfInfiniteRadius}
The question studied in this section is how large we can choose the radius of the four sectors. It will turn out that the absence of singularities outside the point $\infty$ enables us to increase the sectors' radii as far as we like. Hence, we can actually use sectors of infinite radius.

\begin{lemmadefi}\label{defGaussianEnhancedSheaf}
Consider the following set of data:
\begin{itemize}
	\item a finite subset $C\subset \C^\times$,
	\item a generic direction $\theta_0$ with respect to $C$ (which defines an order on $C$),
	\item a family $\ranks=(r_c)_{c\in C}$ of natural numbers $r_c\in \Z_{>0}$,
	\item a family $\sigma=(\sigma_k)_{k\in\Z/4\Z}$ of $(r\times r)$-matrices (where $r\defeq \sum_{c\in C}r_c$) such that $\sigma_1$ and $\sigma_3$ (resp.\ $\sigma_2$ and $\sigma_4$) are upper (resp.\ lower) block-triangular, with the block structure given by the numbers $r_c$ (ordered according to $\theta_0$).
\end{itemize}
Define the sectors $S_k\defeq\{ z\in\C \mid \arg z\in [\theta_0+(k-1)\frac{\pi}{2},\theta_0+k\frac{\pi}{2}] \text{ if $z\neq 0$} \}$, which are closed sectors of infinite radius at $\infty$, but can also be considered as closed sectors (including the vertex) at $0$. As usual, we set $S_{k,k+1}\defeq S_k\cap S_{k+1}$.

Then there exists an enhanced sheaf $\cF_\sigma^{C,\theta_0,\ranks}\in\Modk{\C\times\R}$ (or $\cF_\sigma$ for short) on $\C$ together with isomorphisms $$\alpha_k\colon (\cF_\sigma)_{S_k}\TO{\sim}\pitens{S_k}\displaystyle\expg$$
such that the transition isomorphisms $\alpha_{k+1}\circ \alpha_k^{-1}$, automorphisms of $\pitens{S_{k,k+1}}\expg$, are given by the matrices $\sigma_k$.
Moreover, the sheaf $\cF_\sigma$ thus defined is unique up to unique isomorphism.

If an enhanced sheaf on $\C$ is isomorphic to one of this form, we will call it an \emph{enhanced sheaf of pure Gaussian type}.
\end{lemmadefi}

The following theorem shows that this finally is an enhanced sheaf (on $\C$) describing globally (on $\PP$) the enhanced solutions of $\cM$. (In contrast to the formulation of Lemma \ref{propGlobalEnhSheaf}, we do not write extension by zero.)

\begin{thm}\label{thmBigSectors}
Let $\cM$ be a D-module of pure Gaussian type $C$, $\ranks$ the family of ranks from its Levelt--Turrittin decomposition, $\theta_0$ a generic direction and recall the Stokes multipliers $\sigma=(\sigma_k)_{k\in\Z/4\Z}$ from the previous section. Then there is an isomorphism
$$\SolE{\PP}(\cM)\iso \kE{\PP}\conv \cF^{C,\theta_0,\ranks}_\sigma.$$
\end{thm}

In the proof of this theorem, let us write $\cF\defeq \cF_\sigma^{C,\theta_0,\ranks}$. We make use of the following lemma, which gives an alternative description of $\SolE{\PP}(\cM)$ away from the singularity.
\begin{lemma}\label{lemmaIsoOnB}
Let $B\subset \C$ be a closed ball of finite radius around $0$. Then there is an isomorphism
\begin{equation*}
\pitens{B}\SolE{\PP}(\cM)\iso \pitens{B}\kE{\PP}\conv \cF.
\end{equation*}
\end{lemma}
\begin{proof}
We abbreviate $B_k\defeq B\cap S_k$ and $B_{k,k+1}\defeq B\cap S_{k,k+1}$ for $k\in\Z/4\Z$, as well as $H\defeq \SolE{\PP}(\cM)$ and $\sM\defeq\expg$. 

For any $k\in\Z/4\Z$, we choose the following isomorphism:
$$\pitens{B_k}H\TO{\vartheta}\pitens{B_k}(\kE{\PP})^r\TO{s_k}\pitens{B_k}(\kE{\PP})^r\TO{\tau}\pitens{B_k}\sM.$$
Here, $\vartheta$ is the isomorphism from \cref{GaussBasicProperties}, $\tau$ is the canonical isomorphism (see \cref{boundedExponent}), and $s_1=\1$, $s_2=\sigma_1$, $s_3=\sigma_2\sigma_1$, $s_4=\sigma_3\sigma_2\sigma_1$. With this choice, the transition maps are given by the $\sigma_k$. (Note that at this point we use that the monodromy is trivial). Hence, one can construct the desired isomorphism.
\end{proof}

\begin{proof}[Proof of \cref{thmBigSectors}]
Recall the notations $R$ and $\Sigma_k$ from Section~\ref{sectionStokesMultipliers}. We abbreviate $H\defeq \SolE{\PP}(\cM)$ and $\sM\defeq \expg$. Moreover, we choose $\rho>R$ and set $B\defeq \{ z\in \C\mid |z|\leq \rho \}$, $\Sigma\defeq \bigcup_{k\in\Z/4\Z}\Sigma_k$, $D\defeq \Sigma\cap B$ and $D_k\defeq D\cap \Sigma_k$ (similarly, $D_{k,k+1}\defeq D\cap \Sigma_{k,k+1}$).

Firstly, one uses \eqref{eq:AutosOnFourSectors} to obtain an isomorphism
\begin{equation}\label{eq:isoOnSigma}
\pitens{\Sigma}H\iso \pitens{\Sigma}\kE{\PP}\conv \cF.
\end{equation}

Secondly, we determine an isomorphism
\begin{equation}\label{eq:isoOnB}
\pitens{B}H\iso \pitens{B}\kE{\PP}\conv \cF.
\end{equation}
The existence of such an isomorphism was shown in \cref{lemmaIsoOnB}. However, it is neither canonical nor unique, but depends on the choice of a trivialization $\vartheta\colon \pitens{B}H\iso \pitens{B}(\kE{\PP})^r$. We choose $\vartheta$ in such a way that the composition
\begin{equation*}
\vartheta_1\circ\alpha_1^{-1}\colon \pitens{D_1}\kE{\PP}\conv \sM\TO{\sim} \pitens{D_1}(\kE{\PP})^r
\end{equation*}
is the canonical isomorphism. We can then conclude that \eqref{eq:isoOnSigma} and \eqref{eq:isoOnB} agree on $D$ and the theorem follows.
\end{proof}

The next lemma shows that we can ``deform'' the sectors $S_k$ without crossing a Stokes line and describe $\cF^{C,\theta_0,\ranks}_\sigma$ equivalently.
\begin{lemma}\label{lemmaCompatibilitySectors}
Let $\cS_k$, $k\in\Z/4\Z$, be four closed sectors of infinite radius at $\infty$. Assume that $\cS_k$ contains exactly the same Stokes directions as $S_k$. Then \cref{defGaussianEnhancedSheaf} defines the same sheaf $\cF^{C,\theta_0,\ranks}_\sigma$ if we replace $S_k$ by $\cS_k$.
\end{lemma}

\section[Stokes data and a Riemann--Hilbert correspondence]{Stokes data and a Riemann--Hilbert correspondence for systems of pure Gaussian type}\label{sectionStokesData}
We have reduced to a small set of data necessary for determining a D-module of pure Gaussian type. We use this data to establish a Riemann--Hilbert correspondence for D-modules of pure Gaussian type. In this section, we will not expand on the proofs of equivalences of categories, which are mainly straightforward.

We fix a finite subset $C\subset \C^\times$ as well as a generic direction $\theta_0$ and consider the sectors $S_k=\{ z\in\C\mid \arg z\in [\theta_0+(k-1)\frac{\pi}{2},\theta_0+k\frac{\pi}{2}] \text{ if $z\neq 0$} \}$. We also fix a positive integer $r_c$ for any $c\in C$.

Let $\ModGaussCr$ be the full subcategory of $\ModholD{\PP}$ consisting of objects of pure Gaussian type $C$ and with a Levelt--Turrittin decomposition satisfying $\mathrm{rk}\, \cR_c=r_c$ for every $c\in C$.

Let $\EbGaussCr$ be the full subcategory of $\EbIk{\PP}$ consisting of objects $H$ satisfying $\pitens{\C}H\iso H$ and admitting isomorphisms 
\begin{equation*}
\pitens{S_k}H\iso \pitens{S_k}\Expg
\end{equation*} 
for $k\in\Z/4\Z$. (Note that these isomorphisms are not part of the data.)

\begin{prop}\label{propModGaussEbGauss}
The functor $\SolE{\PP}$ induces an equivalence between $\ModGaussCr$ and $\EbGaussCr$.
\end{prop}
Essential surjectivity follows quickly from the description of the essential image of $\SolE{X}$ by T.\ Mochizuki (see \cite[Lemma 9.8]{Moc16}) and the comparison between enhanced exponentials \cite[Corollary 5.2.3]{DK18} (cf. also \cite[Lemma 5.15]{Moc16}).

The results of the previous sections enable us to describe the objects of $\EbGaussCr$ in terms of linear algebra data. Choose a numbering of the elements of $C$ such that $\pc{1}<_{\theta_0} \pc{2}<_{\theta_0}\ldots<_{\theta_0} \pc{n}$. We will write $r_j$ instead of $r_{\pc{j}}$. 

\begin{defi}\label{defStokesData}
One defines the \emph{category $\StokesCr$ of Stokes data of pure Gaussian type $(C,\theta_0,(r_c)_{c\in C})$} as follows:
\begin{itemize}
\item An object $\sigma=(\sigma_k)_{k\in\Z/4\Z}\in\Ob\StokesCr$ is a family of four block matrices with the properties:
\begin{itemize}
\item The block structure is given by the numbers $r_j$ ($j\in\{1,\ldots,n\}$), i.e.\ the $j$th diagonal block has size $r_j\times r_j$.
\item The matrices $\sigma_1$ and $\sigma_3$ are upper block-triangular and the matrices $\sigma_2$ and $\sigma_4$ are lower block-triangular.
\item The matrix $\sigma_k$ is invertible for any $k\in\Z/4\Z$. (With the above properties, this is equivalent to saying that the blocks along the diagonal are invertible.)
\item The product of the $\sigma_k$ is the identity: $\sigma_4\sigma_3\sigma_2\sigma_1=\1$.
\end{itemize}
\item A morphism $\delta=(\delta_k)_{k\in\Z/4\Z}\in\Hom{\StokesCr}(\sigma,\widetilde{\sigma})$ between two objects $\sigma=(\sigma_k)_{k\in\Z/4\Z}$ and $\widetilde{\sigma}=(\widetilde\sigma_k)_{k\in\Z/4\Z}$ is a family of four block matrices with the properties:
\begin{itemize}
\item The block structure is given by the numbers $r_j$ ($j\in\{1,\ldots,n\}$).
\item The matrix $\delta_k$ is block-diagonal for every $k\in\Z/4\Z$.
\item For any $k\in\Z/4\Z$, one has $\widetilde{\sigma}_k\delta_k=\delta_{k+1}\sigma_k$.
\end{itemize}
Composition of morphisms is given by matrix multiplication.
\end{itemize}
\end{defi}

\begin{rem}
Let us give an explanation of how one could think of objects and morphisms in the category of Stokes data $\StokesCr$. This also gives an idea for making a link with the description of Stokes data in \cite{Sab16}.

An object consists of four matrices which will correspond to the Stokes matrices describing the transition between the four sectors. We can therefore imagine them to be arranged in a ``circle'', i.e.\ a diagram of the form
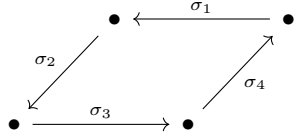
\begin{center}
\begin{tikzcd}
& \bullet \arrow{dl}[swap]{\sigma_2} &[-10pt] & \bullet \arrow{ll}[swap]{\sigma_1} \\[+10pt]
\bullet\arrow{rr}{\sigma_3} & & \bullet\arrow{ur}[swap]{\sigma_4}
\end{tikzcd}
\end{center}
One can think of the vertices $\bullet$ as vector spaces $\kk^r=\bigoplus_{j=1}^n \kk^{r_j}$ which (by the given grading) have two natural filtrations: The filtration $F_m\kk^r=\bigoplus_{j=1}^m\kk^{r_j}$ is respected by the matrices $\sigma_1$ and $\sigma_3$, whereas the filtration $F'_m\kk^r=\bigoplus_{j=n-m+1}^n \kk^{r_j}$ is respected by the matrices $\sigma_2$ and $\sigma_4$.

A morphism between two such diagrams can then be visualized as
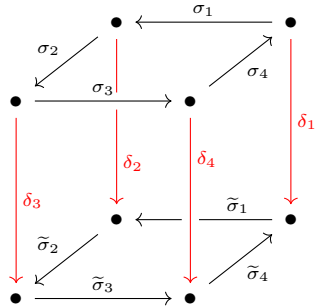
\begin{center}
\begin{tikzcd}
& \bullet \arrow{dl}[swap]{\sigma_2}\arrow[red]{dd}[near end]{\delta_2} &[-10pt] & \bullet \arrow{ll}[swap]{\sigma_1}\arrow[red]{dd}{\delta_1} \\
\bullet\arrow[crossing over]{rr}{\sigma_3}\arrow[red]{dd}{\delta_3} & & \bullet\arrow{ur}[swap]{\sigma_4}\\[+15pt]
& \bullet \arrow{dl}[swap]{\widetilde\sigma_2} &[-10pt] & \bullet \arrow{ll}[swap,near start]{\widetilde\sigma_1} \\
\bullet\arrow{rr}{\widetilde\sigma_3} & & \bullet\arrow{ur}[swap]{\widetilde\sigma_4}\arrow[red, crossing over, <-]{uu}[near end, swap]{\delta_4}
\end{tikzcd}
\end{center}
and the relations required in \cref{defStokesData} amount to saying that this diagram is commutative. The matrices $\delta_k$ respect the grading $\kk^r=\bigoplus_{j=1}^n \kk^{r_j}$, i.e.\ they are compatible with both filtrations considered above.

An intuitive reason why the $\sigma_k$ are block-triangular, while the $\delta_k$ need to be block-diagonal is the following: The matrices $\sigma_k$ are the transition matrices, which means that they describe isomorphisms on the boundaries of the sectors, where one has a well-defined ordering of the parameters $\pc{j}$ (cf.\ \cref{autosAsMatrices}). In contrast, the $\delta_k$ are meant to describe morphisms on the sectors $S_k$, where no pair of parameters has a global well-defined order. Therefore, $\delta_k$ must be compatible with any order of the $\pc{j}$.
\end{rem}

\begin{prop}\label{propStokesDataEnhGauss}
The functor 
$$\StokesCr\to\EbGaussCr,\quad \sigma=(\sigma_k)_{k\in\Z/4\Z}\mapsto \kE{\PP}\conv \cF_{\sigma}$$
is an equivalence of categories, where $\cF_\sigma$ is as in \cref{sectionClosedSectorsOfInfiniteRadius}.
\end{prop}

\begin{cor}\label{equivalenceDmodStokesData}
There is an equivalence of categories $\ModGaussCr \TO{\sim} \StokesCr$.
\end{cor}
The corresponding functor assigns to a D-module of pure Gaussian type $\cM$ its Stokes matrices with respect to the generic direction $\theta_0$.

\section{Analytic and topological Fourier--Laplace transform}\label{sectionFourier}
Classically, for a module $M$ over the Weyl algebra $\C[z]\langle \partial_z\rangle$, the Fourier--Laplace transform $\widehat{M}$ is the $\C[w]\langle \partial_w\rangle$-module defined as follows: As a set, we have $\widehat{M}=M$, and the structure of a $\C[w]\langle \partial_w\rangle$-module is defined by $w\cdot m\defeq \partial_z m$ and $\partial_w m\defeq -z\cdot m$. The corresponding integral transform is given as follows (see \cite{KL}).

Consider the projections
\begin{center}
\begin{tikzcd}
&\PP_z\times\PP_w\arrow{ld}[swap]{p_z}\arrow{rd}{p_w}\\
\PP_z & & \PP_w
\end{tikzcd}
\end{center}
where $\PP_z$ denotes the complex projective line with affine coordinate $z$ in the chart $\C_z\subset \PP_z$ at $0$, and similarly for $\PP_w$.

\begin{defi}
Let $\cM\in\DbD{\PP_z}$. We define the \emph{Fourier--Laplace transform} $\dF{\cM}$ of $\cM$ by
$$\dF{\cM}\defeq \DD {p_w}_*\big(\ExpD{-zw}{\C_z\times\C_w}{\PP_z\times\PP_w}\tensD \DD p_z^*\cM\big)\in \DbD{\PP_w}.$$
This defines a functor $\dF{(\bullet)}\colon \DbD{\PP_z}\to\DbD{\PP_w}$.
\end{defi}

In the same spirit, one can define a transform for enhanced ind-sheaves (see \cite{KS16b}) and enhanced sheaves. Consider the projections
\begin{center}
\begin{tikzcd}
&\C_z\times\C_w\times\R\arrow{ld}[swap]{\widetilde{p}}\arrow{rd}{\widetilde{q}}\\
\C_z\times\R & & \C_w\times\R
\end{tikzcd}
\end{center}
\begin{defi}
Let $\cF\in\Dbk{\C_z\times\R}$ be an enhanced sheaf. We define its \emph{enhanced Fourier--Sato transform} $\sF{\cF}$ by
$$\sF{\cF}\defeq \RR \qtild_!\big(\Exps{-\real{zw}}{\C_z\times\C_w}{\C_z\times\C_w}\convs \ptild^{\mspace{2mu}-1}\cF\big)[1]\in\Dbk{\C_w\times\R}.$$
This defines a functor $\sF{(\bullet)}\colon \Dbk{\C_z\times\R}\to \Dbk{\C_w\times\R}$.
\end{defi}

An important observation on our way to describing the Fourier--Laplace transform of a D-module of pure Gaussian type is the compatibility of these transformations with the enhanced solution functor (cf.\ \cite[Theorem 4.17]{KS16b}).

\begin{lemma}\label{lemmaCompatibilityFourierSol}
Let $\cM\in \ModholD{\PP_z}$ and $\SolE{\PP_z}(\cM)\iso \kE{\PP_z}\conv \cF$ for some $\cF\in \Modk{\C_z\times\R}$. One has an isomorphism in $\EbIk{\PP_w}$
$$\SolE{\PP_w}(\dF{\cM})\iso \kE{\PP_w} \conv \sF{\cF}.$$
\end{lemma}

\section{Aligned parameters}\label{sectionAligned}
In \cite{Sab16}, C.\ Sabbah treated the case of a D-module of pure Gaussian type $C$ with $\arg c=\arg d$ for any $c,d\in C$, i.e.\ the parameter set $C$ is ``aligned'' along a half-line through the origin.

\subsection{Main statement}\label{MainStatementAligned}
Let $\cM\in\DbholD{\PP}$ be a D-module of pure Gaussian type $C$, where all the elements of $C$ have the same argument $\arg C$.

The directions of the Stokes lines are: $-\frac{\pi}{4}-\frac{1}{2}\arg C+k\frac{\pi}{2}$, $k\in\Z/4\Z$.
(These values are the same for any pair $c,d\in C$.) In particular, we can choose $\theta_0\defeq -\frac{1}{2}\arg C$ as a generic direction. Note that this involves a choice of $\frac{1}{2}\arg C$, and we choose $\theta_0\in [-\frac{\pi}{2},\frac{\pi}{2})$.

It is known from \cref{thmBigSectors} that
$$\SolE{\PP}(\cM)\iso \kE{\PP}\conv \cF_\sigma^{C,\theta_0,\ranks}.$$
Therefore, in view of \cref{lemmaCompatibilityFourierSol}, the main step in computing the Fourier--Laplace transform of $\cM$ topologically is the proof of the following statement. Let $C$ and $\theta_0$ be as above. Let $r_c\in\Z_{>0}$ be a positive integer for any $c\in C$, and let $\sigma=(\sigma_k)_{k\in\Z/4\Z}$ be a family of four block matrices (with block structure induced be the numbering on $C$ with respect to $\theta_0$) such that $\sigma_k$ is upper (resp.\ lower) block-triangular for $k$ odd (resp.\ even) and $\sigma_4\sigma_3\sigma_2\sigma_1=\1$.
\begin{thm}\label{thmAlignedParametersEnhancedSheaves}
Let $C$, $\theta_0$, $\ranks$ and $\sigma$ be as above. We set $\widehat{C}\defeq -1/C= \{ -\frac{1}{c}\mid c\in C \}$, $\widehat{\theta}_0\defeq \pi-\theta_0$ and $\widehat{\ranks}\defeq (r_{-\frac{1}{\widehat{c}}})_{\widehat{c}\in\widehat{C}}$. Then there is an isomorphism
$$\sF{\cF_\sigma^{C,\theta_0,\ranks}}\iso \cF_\sigma^{\widehat{C},\widehat{\theta}_0,\widehat{\ranks}}.$$
In particular, the gluing matrices $\sigma=(\sigma_k)_{k\in\Z/4\Z}$ remain the same (although sectors and exponential factors change).
\end{thm}
As a corollary, we obtain the following result, which was already obtained in the context of Stokes data attached to Stokes-filtered local systems by C.\ Sabbah (cf.\ \cite[Lemma 1.4, Theorem 4.2]{Sab16}). The statement is illustrated in \cref{figureMainResultAligned}.

\begin{figure}[ht]
\centering
\begin{picture}(320,150)
\put(25,15){
\begin{tikzpicture}[scale=0.5]
\draw [teal, ultra thick] (0,0) -- (4,-0.94427191); 
\draw [teal] (0,0) -- (-0.94427191,-4);
\draw [teal] (0,0) -- (-4,0.94427191);
\draw [teal] (0,0) -- (0.94427191,4);
\path [fill=olive, fill opacity=0.25] (0,0) to (4,-0.94427191) to (4,-4) -- (-0.94427191,-4) -- (0,0);
\path [fill=green, fill opacity=0.35] (0,0) to (-0.94427191,-4) to (-4,-4) -- (-4,0.94427191) -- (0,0);
\path [fill=orange, fill opacity=0.35] (0,0) to (-4,0.94427191) to (-4,4) -- (0.94427191,4) -- (0,0);
\path [fill=yellow, fill opacity=0.35] (0,0) to (0.94427191,4) to (4,4) -- (4,-0.94427191) -- (0,0);
\draw [->, thick, red, dotted] (0,0) -- (-0.3*4,-0.3*2.472135955); 
\draw [->, thick, red, dotted] (0,0) -- (-0.3*2.472135955,0.3*4);
\draw [->, thick, red,dotted] (0,0) -- (0.3*4,0.3*2.472135955);
\draw [->, thick, red, dotted] (0,0) -- (0.3*2.472135955,-0.3*4);
\node [gray] (A) at (3,3) {$S_1$};
\node [gray] (B) at (-3,3) {$S_2$};
\node [gray] (C) at (-3,-3) {$S_3$};
\node [gray] (D) at (3,-3) {$S_4$};

\draw [<-,thick,dashed] (-0.65*3.69074767,-0.65*0.2198235) to [out=120,in=210] (-0.65*3.20279701,0.65*1.84716865);
\draw [<-,thick,dashed] (-0.65*0.2198235,0.65*3.69074767) to [out=30,in=130] (0.65*1.84716865,0.65*3.20279701);
\draw [<-,thick,dashed] (0.65*3.69074767,0.65*0.2198235) to [out=-60,in=40] (0.65*3.20279701,-0.65*1.84716865);
\draw [<-,thick,dashed] (0.65*0.2198235,-0.65*3.69074767) to [out=-160,in=-50] (-0.65*1.84716865,-0.65*3.20279701);
\node (E) at (3,-0.4) {$\sigma_4$};
\node (F) at (0.25,2.9) {$\sigma_1$};
\node (G) at (-3,0.3) {$\sigma_2$};
\node (H) at (-0.25,-2.9) {$\sigma_3$};

\node [teal, thick] (I) at (3.6,-1.3) {$\theta_0$};
\end{tikzpicture}
}
\put(50,0){
\centering
\small Original system
}
\put(165,15){
\begin{tikzpicture}[scale=0.5]
\draw [teal, ultra thick] (0,0) -- (-4,-0.94427191); 
\draw [teal] (0,0) -- (0.94427191,-4);
\draw [teal] (0,0) -- (4,0.94427191);
\draw [teal] (0,0) -- (-0.94427191,4);
\path [fill=yellow, fill opacity=0.35] (0,0) to (-4,-0.94427191) to (-4,-4) -- (0.94427191,-4) -- (0,0);
\path [fill=orange, fill opacity=0.35] (0,0) to (0.94427191,-4) to (4,-4) -- (4,0.94427191) -- (0,0);
\path [fill=green, fill opacity=0.35] (0,0) to (4,0.94427191) to (4,4) -- (-0.94427191,4) -- (0,0);
\path [fill=olive, fill opacity=0.25] (0,0) to (-0.94427191,4) to (-4,4) -- (-4,-0.94427191) -- (0,0);
\draw [->, thick, red, dotted] (0,0) -- (0.3*4,-0.3*2.472135955); 
\draw [->, thick, red, dotted] (0,0) -- (0.3*2.472135955,0.3*4);
\draw [->, thick, red,dotted] (0,0) -- (-0.3*4,0.3*2.472135955);
\draw [->, thick, red, dotted] (0,0) -- (-0.3*2.472135955,-0.3*4);
\node [gray] (A) at (-3,3) {$\widehat{S_4}$};
\node [gray] (B) at (3,3) {$\widehat{S_3}$};
\node [gray] (C) at (3,-3) {$\widehat{S_2}$};
\node [gray] (D) at (-3,-3) {$\widehat{S_1}$};

\draw [->,thick,dashed] (0.65*3.69074767,-0.65*0.2198235) to [out=60,in=-40] (0.65*3.20279701,0.65*1.84716865);
\draw [->,thick,dashed] (0.65*0.2198235,0.65*3.69074767) to [out=150,in=50] (-0.65*1.84716865,0.65*3.20279701);
\draw [->,thick,dashed] (-0.65*3.69074767,0.65*0.2198235) to [out=-120,in=140] (-0.65*3.20279701,-0.65*1.84716865);
\draw [->,thick,dashed] (-0.65*0.2198235,-0.65*3.69074767) to [out=-40,in=-110] (0.65*1.84716865,-0.65*3.20279701);
\node (E) at (-2.8,-1.15) {$\sigma_4$};
\node (F) at (-1.1,2.85) {$\sigma_3$};
\node (G) at (2.8,1.2) {$\sigma_2$};
\node (H) at (1.15,-2.9) {$\sigma_1$};

\node [teal, thick] (I) at (-3.6,-0.25) {$\widehat\theta_0$};
\end{tikzpicture}
}
\put(172,0){
\small Fourier--Laplace transform
}
\end{picture}
\caption{The complex plane covered by four closed sectors, which are determined by the generic directions $\theta_0$ and $\widehat\theta_0=\pi-\theta_0$. (The red arrows indicate the Stokes directions.)\newline
If a D-module of pure Gaussian type has a Hukuhara--Turrittin decomposition on each of the sectors $S_k$ (on the left) with exponents $-\frac{c}{2}z^2$ and Stokes multipliers $\sigma_k$, then its Fourier--Laplace transform has a Hukuhara--Turrittin decomposition on the sectors $\widehat{S}_k$ (on the right) with exponents $\frac{1}{2c}w^2$ and Stokes multipliers $\widehat{\sigma}_k=\sigma_k$.}
\label{figureMainResultAligned}
\end{figure}

\begin{cor}\label{corAlignedParametersDmodules}
Let $C\subset \C^\times$ be a finite subset whose elements have constant argument $\arg C$. Let $\cM\in\DbholD{\PP}$ be of pure Gaussian type $C$ and let $(\sigma_k)_{k\in\Z/4\Z}$ be Stokes multipliers with respect to the generic direction $\theta_0=-\frac{1}{2}\arg C$. Then the Fourier--Laplace transform $\dF{\cM}$ of $\cM$ is of pure Gaussian type $\widehat{C}=-1/C$ and Stokes multipliers with respect to the generic direction $\widehat{\theta}_0\defeq \pi-\theta_0$ are given by $(\sigma_k)_{k\in\Z/4\Z}$.
\end{cor}

The rest of this section will be concerned with the proof of \cref{thmAlignedParametersEnhancedSheaves}.
The idea of the proof is as follows: We choose a decomposition of the plane into four closed sectors $\cS_k$, $k\in\Z/4\Z$, on which $\cF=\cF_\sigma^{C,\theta_0,\ranks}$ is trivialized as a direct sum of exponential enhanced sheaves. As usual, write $\cS_{k,k+1}\defeq \cS_k\cap \cS_{k+1}$. We will first compute the enhanced Fourier--Sato transforms of these exponential enhanced ind-sheaves on the $\cS_k$ and $\cS_{k,k+1}$ (and hence $\sF{(\cF_{\cS_k})}$ and $\sF{(\cF_{\cS_{k,k+1}})}$). Setting $\cH_+\defeq \cS_1\cup \cS_2$, $\cH_-\defeq \cS_3\cup \cS_4$  and $L\defeq \cS_{41}\cup \cS_{23}$, we can model the gluing of $\cF$ from the restrictions to sectors in terms of short exact sequences in $\Modk{\C\times\R}$:
\begin{align*}
0\To \cF_{\cH_+}\To \cF_{\cS_1}&\oplus \cF_{\cS_2}\To \cF_{\cS_{12}}\To 0\\
0\To \cF_{\cH_-}\To \cF_{\cS_3}&\oplus \cF_{\cS_4}\To \cF_{\cS_{34}}\To 0\\
0\To \cF_{L}\To \cF_{\cS_{41}}&\oplus \cF_{\cS_{23}}\To \cF_{\{0\}}\To 0\\
0\To \cF\To \cF_{\cH_+}&\oplus \cF_{\cH_-}\To \cF_{L}\To 0
\end{align*}
Applying the enhanced Fourier--Sato transform, we obtain distinguished triangles (which will turn out to be just short exact sequences), and we can determine step by step the enhanced Fourier--Sato transforms of $\cF_{H_+}$, $\cF_{H_-}$, $\cF_{L}$, and finally of $\cF$.

We will give a proof for the case where $\arg C\in(-\frac{\pi}{2},\frac{\pi}{2})$, i.e.\ $\real{c}>0$. The arguments for the other cases $\real{c}<0$ and $\real{c}=0$ work completely along the same lines. However, the geometry of the objects involved depends on the sign of $\real{c}$.

\subsection{Exponential enhanced sheaves on closed sectors}\label{sectionFourierOfExponentialsOnSectors}
Let us choose the sectors (note that $0\in\cS_k$)
\begin{align*}
\cS_1&\defeq \left\{ z\in \C_z\,\left| \, \arg z \in \left[0,\frac{\pi}{2}-\arg C\right]\text{if $z\neq 0$} \right\} \right.,\\
\cS_2&\defeq \left\{ z\in \C_z\,\left| \, \arg z \in \left[\frac{\pi}{2}-\arg C, \pi \right]\text{if $z\neq 0$} \right\} \right.,\\
\cS_3&\defeq \left\{ z\in \C_z\,\left| \, \arg z \in \left[-\pi,-\frac{\pi}{2}-\arg C \right]\text{if $z\neq 0$} \right\} \right.,\\
\cS_4&\defeq \left\{ z\in \C_z\,\left| \, \arg z \in \left[-\frac{\pi}{2}-\arg C, 0  \right]\text{if $z\neq 0$} \right\} \right..
\end{align*}
Denote the half-lines bounding the sectors by $\cS_{k,k+1}\defeq \cS_k \cap \cS_{k+1}$ for $k\in\Z/4\Z$.
It is easy to check that each of these sectors contains exactly one Stokes direction and that they are compatible with the $S_k$ in the sense of \cref{lemmaCompatibilitySectors}.

The first aim is to compute the enhanced Fourier--Sato transforms of the enhanced exponentials $\Exps{-\real{\frac{c}{2}z^2}}{\cS_k}{\C_z}$, which are the building blocks of $\cF^{C,\theta_0,\ranks}_\sigma$ on sectors. As mentioned, we assume $c=\rec+i\imc\in \C^\times$ with $\rec>0$. We will give the proof for $k=1$.

We can compute
\begin{align}
\sF{\Exps{-\real{\frac{c}{2}z^2}}{\cS_1}{\C_z}}&=\RR \qtild_! \Big(\kk_{\{t-\real{zw}\geq 0\}}\convs \ptild^{\mspace{2mu}-1}\big(\pitens{\cS_1} \kk_{\{t-\real{\frac{c}{2}z^2}\geq 0\}}\big)\Big)[1]\label{eq:enhFourierOfExpSheaf}\\
&\iso \RR\qtild_! \Big(\pitens{\cS_1\times\C_w} \big(\kk_{\{ t-\real{zw}\geq 0 \}}\convs \kk_{\{t-\real{\frac{c}{2}z^2}\geq 0\}}\big)\Big)[1]\notag\\
&\iso \RR\qtild_! \kk_{\{ (z,w,t)\in\C_z\times\C_w\times\R \mid z\in \cS_1, t-\real{(zw+\frac{c}{2}z^2)}\geq 0 \}}[1].\notag
\end{align}
In particular, the stalks of the cohomology sheaves at a point $(\wf,\tf)\in \C_w\times\R$ are determined by the topology of the intersection of two subspaces of $\C_z$:
\begin{equation*}
\mathrm{H}^l\big(\sF{\Exps{-\real{\frac{c}{2}z^2}}{\cS_1}{\C_z}}\big)_{(\wf,\tf)}\iso \Hc{l+1}\big(\cS_1\cap \big\{z\in\C_z\mid \tf - \real{\big(z\wf+\frac{c}{2}z^2\big)}\geq 0\big\},\kk\big).
\end{equation*}
The inequality $\tf - \real{\big(z\wf+\frac{c}{2}z^2\big)}\geq 0$ describes a region bounded by the two branches of a hyperbola. The hyperbola can be written in standard form if we write $z=z_1+iz_2$ and apply the coordinate transform \begin{equation}\label{eq:coordinateTransform}  x_1\defeq z_1-\frac{\imc}{\rec}z_2+\frac{\rewf}{\rec}, \quad x_2\defeq z_2+\frac{\rec\imwf-\imc\rewf}{|c|^2}.\end{equation}

Then, the space to be considered is the intersection of the (hyperbolic) region given by
$$\frac{\rec}{2}x_1^2 - \frac{|c|^2}{2\rec}x_2^2\leq \tf + \real{\frac{1}{2c}\wf^2}$$
and the sector given by
$$x_1\geq \frac{\wf_1}{\rec}, \quad x_2\geq\frac{\rec\wf_2-\imc\wf_1}{|c|^2}.$$
Clearly, the topology of this intersection highly depends on the values of $\tf$, $\rewf$ and $\imwf$. It is easy to see that the above compactly supported cohomology groups are trivial unless the intersection has a compact connected component (see \cref{figureHyperbolaAndSectorAligned}, noting that the unbounded components have vanishing cohomology with compact support), and by elementary considerations one can determine the cases in which such a compact connected component exists. This yields the following lemma.
\begin{lemma} There are isomorphisms
\begin{equation}\label{eq:CohomOfHyperbolaCapSectorDegNot-1}
\mathrm{H}^l\big( \sF{\Exps{-\real{\frac{c}{2}z^2}}{\cS_1}{\C_z}} \big)\iso0\text{ for $l\neq -1$}
\end{equation}
and 
\begin{equation}\label{eq:CohomOfHyperbolaCapSectorDeg-1}
\mathrm{H}^{-1}\big( \sF{\Exps{-\real{\frac{c}{2}z^2}}{\cS_1}{\C_z}} \big)_{(\wf,\tf)} \iso \begin{cases} \kk & \text{if $\imc\wf_1-\rec\wf_2\geq 0$ and $-\phirp{c}(\wf)\leq \tf< -\phirm{c}(\wf)$} \\ 0 &\text{otherwise}\end{cases}
\end{equation}
with the continuous functions $\phirp{c},\phirm{c}\colon \C_w\to\R$ defined by
$$\phirp{c}(\rew+i\imw)\defeq\begin{cases} \frac{\rew^2}{2\rec} &\text{if $\rew\leq0$}\\ 0 & \text{if $\rew> 0$} \end{cases}$$
and
$$\phirm{c}(\rew+i\imw)\defeq\begin{cases} \frac{1}{2|c|^2}(\rec\rew^2-\rec\imw^2+2\imc \rew\imw)=\real{\frac{1}{2c}w^2} &\text{if $\rew\leq 0$}\\ -\frac{1}{2\rec|c|^2}(\imc\rew-\rec\imw)^2 \eqdef \dif{c}(w) & \text{if $\rew> 0$} \end{cases}.$$
Observe that $\phirp{c}(w)-\phirm{c}(w)=\frac{1}{2\rec|c|^2}(\imc\rew-\rec\imw)^2$, so $\phirp{c}(w)\geq\phirm{c}(w)$ for all $w\in \C_w$.
\end{lemma}

\begin{figure}[ht]
\centering
\begin{picture}(350,70)
\put(-10,0){
\newcommand{\pt}{(-3/25)}
\newcommand{\pcr}{7}
\newcommand{\pci}{1}
\newcommand{\pwr}{(-2)}
\newcommand{\pwi}{(-2)}
\begin{tikzpicture}[scale=1.125]
\draw[Cyan, domain=-1:1] plot [smooth] (\x,{sqrt(2*\pcr/(\pcr*\pcr+\pci*\pci)*(-\pt-1/2/(\pcr*\pcr+\pci*\pci)*(\pcr*\pwr*\pwr-\pcr*\pwi*\pwi+2*\pci*\pwr*\pwi)+\pcr/2*\x*\x))});
\fill [Cyan, opacity=0.5, domain=-1:1] (-1,{sqrt(2*\pcr/(\pcr*\pcr+\pci*\pci)*(-\pt-1/2/(\pcr*\pcr+\pci*\pci)*(\pcr*\pwr*\pwr-\pcr*\pwi*\pwi+2*\pci*\pwr*\pwi)+\pcr/2))}) -- plot [smooth] (\x,{sqrt(2*\pcr/(\pcr*\pcr+\pci*\pci)*(-\pt-1/2/(\pcr*\pcr+\pci*\pci)*(\pcr*\pwr*\pwr-\pcr*\pwi*\pwi+2*\pci*\pwr*\pwi)+\pcr/2*\x*\x))}) -- (1,{sqrt(2*\pcr/(\pcr*\pcr+\pci*\pci)*(-\pt-1/2/(\pcr*\pcr+\pci*\pci)*(\pcr*\pwr*\pwr-\pcr*\pwi*\pwi+2*\pci*\pwr*\pwi)+\pcr/2))}) -- cycle;

\draw[Cyan, domain=-1:1] plot [smooth] (\x,{-sqrt(2*\pcr/(\pcr*\pcr+\pci*\pci)*(-\pt-1/2/(\pcr*\pcr+\pci*\pci)*(\pcr*\pwr*\pwr-\pcr*\pwi*\pwi+2*\pci*\pwr*\pwi)+\pcr/2*\x*\x))});
\fill [Cyan, opacity=0.5, domain=-1:1] (-1,{-sqrt(2*\pcr/(\pcr*\pcr+\pci*\pci)*(-\pt-1/2/(\pcr*\pcr+\pci*\pci)*(\pcr*\pwr*\pwr-\pcr*\pwi*\pwi+2*\pci*\pwr*\pwi)+\pcr/2))}) -- plot [smooth] (\x,{-sqrt(2*\pcr/(\pcr*\pcr+\pci*\pci)*(-\pt-1/2/(\pcr*\pcr+\pci*\pci)*(\pcr*\pwr*\pwr-\pcr*\pwi*\pwi+2*\pci*\pwr*\pwi)+\pcr/2*\x*\x))}) -- (1,{-sqrt(2*\pcr/(\pcr*\pcr+\pci*\pci)*(-\pt-1/2/(\pcr*\pcr+\pci*\pci)*(\pcr*\pwr*\pwr-\pcr*\pwi*\pwi+2*\pci*\pwr*\pwi)+\pcr/2))}) -- cycle;
		
\draw[Red,-] ({\pwr/\pcr},{(\pcr*\pwi-\pci*\pwr)/(\pcr*\pcr+\pci*\pci)})--(1,{(\pcr*\pwi-\pci*\pwr)/(\pcr*\pcr+\pci*\pci)});
\draw[Red,-] ({\pwr/\pcr},{(\pcr*\pwi-\pci*\pwr)/(\pcr*\pcr+\pci*\pci)})--({\pwr/\pcr},{sqrt(2*\pcr/(\pcr*\pcr+\pci*\pci)*(-\pt-1/2/(\pcr*\pcr+\pci*\pci)*(\pcr*\pwr*\pwr-\pcr*\pwi*\pwi+2*\pci*\pwr*\pwi)+\pcr/2))});
\fill[Red, opacity=0.3] ({\pwr/\pcr},{(\pcr*\pwi-\pci*\pwr)/(\pcr*\pcr+\pci*\pci)})--(1,{(\pcr*\pwi-\pci*\pwr)/(\pcr*\pcr+\pci*\pci)})--(1,{sqrt(2*\pcr/(\pcr*\pcr+\pci*\pci)*(-\pt-1/2/(\pcr*\pcr+\pci*\pci)*(\pcr*\pwr*\pwr-\pcr*\pwi*\pwi+2*\pci*\pwr*\pwi)+\pcr/2))})--({\pwr/\pcr},{sqrt(2*\pcr/(\pcr*\pcr+\pci*\pci)*(-\pt-1/2/(\pcr*\pcr+\pci*\pci)*(\pcr*\pwr*\pwr-\pcr*\pwi*\pwi+2*\pci*\pwr*\pwi)+\pcr/2))});
\end{tikzpicture}
}
\put(70,0){
\newcommand{\pt}{5}
\newcommand{\pcr}{1}
\newcommand{\pci}{1}
\newcommand{\pwr}{1.5}
\newcommand{\pwi}{-4}
\begin{tikzpicture}[scale=0.3]
\draw[Cyan, domain=-5:5] plot [smooth] (\x,{sqrt(2*\pcr/(\pcr*\pcr+\pci*\pci)*(-\pt-1/2/(\pcr*\pcr+\pci*\pci)*(\pcr*\pwr*\pwr-\pcr*\pwi*\pwi+2*\pci*\pwr*\pwi)+\pcr/2*\x*\x))});
\fill [Cyan, opacity=0.5, domain=-5:5] (-5,{sqrt(2*\pcr/(\pcr*\pcr+\pci*\pci)*(-\pt-1/2/(\pcr*\pcr+\pci*\pci)*(\pcr*\pwr*\pwr-\pcr*\pwi*\pwi+2*\pci*\pwr*\pwi)+\pcr/2*25))}) -- plot [smooth] (\x,{sqrt(2*\pcr/(\pcr*\pcr+\pci*\pci)*(-\pt-1/2/(\pcr*\pcr+\pci*\pci)*(\pcr*\pwr*\pwr-\pcr*\pwi*\pwi+2*\pci*\pwr*\pwi)+\pcr/2*\x*\x))}) -- (5,{sqrt(2*\pcr/(\pcr*\pcr+\pci*\pci)*(-\pt-1/2/(\pcr*\pcr+\pci*\pci)*(\pcr*\pwr*\pwr-\pcr*\pwi*\pwi+2*\pci*\pwr*\pwi)+\pcr/2*25))}) -- cycle;

\draw[Cyan, domain=-5:5] plot [smooth] (\x,{-sqrt(2*\pcr/(\pcr*\pcr+\pci*\pci)*(-\pt-1/2/(\pcr*\pcr+\pci*\pci)*(\pcr*\pwr*\pwr-\pcr*\pwi*\pwi+2*\pci*\pwr*\pwi)+\pcr/2*\x*\x))});
\fill [Cyan, opacity=0.5, domain=-5:5] (-5,{-sqrt(2*\pcr/(\pcr*\pcr+\pci*\pci)*(-\pt-1/2/(\pcr*\pcr+\pci*\pci)*(\pcr*\pwr*\pwr-\pcr*\pwi*\pwi+2*\pci*\pwr*\pwi)+\pcr/2*25))}) -- plot [smooth] (\x,{-sqrt(2*\pcr/(\pcr*\pcr+\pci*\pci)*(-\pt-1/2/(\pcr*\pcr+\pci*\pci)*(\pcr*\pwr*\pwr-\pcr*\pwi*\pwi+2*\pci*\pwr*\pwi)+\pcr/2*\x*\x))}) -- (5,{-sqrt(2*\pcr/(\pcr*\pcr+\pci*\pci)*(-\pt-1/2/(\pcr*\pcr+\pci*\pci)*(\pcr*\pwr*\pwr-\pcr*\pwi*\pwi+2*\pci*\pwr*\pwi)+\pcr/2*25))}) -- cycle;

\draw[Red,-] ({\pwr/\pcr},{(\pcr*\pwi-\pci*\pwr)/(\pcr*\pcr+\pci*\pci)})--(5,{(\pcr*\pwi-\pci*\pwr)/(\pcr*\pcr+\pci*\pci)});
\draw[Red,-] ({\pwr/\pcr},{(\pcr*\pwi-\pci*\pwr)/(\pcr*\pcr+\pci*\pci)})--({\pwr/\pcr},{sqrt(2*\pcr/(\pcr*\pcr+\pci*\pci)*(-\pt-1/2/(\pcr*\pcr+\pci*\pci)*(\pcr*\pwr*\pwr-\pcr*\pwi*\pwi+2*\pci*\pwr*\pwi)+\pcr/2*25))});
\fill[Red, opacity=0.3] ({\pwr/\pcr},{(\pcr*\pwi-\pci*\pwr)/(\pcr*\pcr+\pci*\pci)})--(5,{(\pcr*\pwi-\pci*\pwr)/(\pcr*\pcr+\pci*\pci)})--(5,{sqrt(2*\pcr/(\pcr*\pcr+\pci*\pci)*(-\pt-1/2/(\pcr*\pcr+\pci*\pci)*(\pcr*\pwr*\pwr-\pcr*\pwi*\pwi+2*\pci*\pwr*\pwi)+\pcr/2*25))})--({\pwr/\pcr},{sqrt(2*\pcr/(\pcr*\pcr+\pci*\pci)*(-\pt-1/2/(\pcr*\pcr+\pci*\pci)*(\pcr*\pwr*\pwr-\pcr*\pwi*\pwi+2*\pci*\pwr*\pwi)+\pcr/2*25))});
\end{tikzpicture}
}
\put(170,0){
\newcommand{\pt}{(200/328)}
\newcommand{\pcr}{4}
\newcommand{\pci}{(-5)}
\newcommand{\pwr}{3}
\newcommand{\pwi}{(-9)}
\begin{tikzpicture}[scale=1.125]
\draw[Cyan] ({sqrt(2*(\pt+1/2/(\pcr*\pcr+\pci*\pci)*(\pcr*\pwr*\pwr-\pcr*\pwi*\pwi+2*\pci*\pwr*\pwi))/\pcr+(\pcr*\pcr+\pci*\pci)/2/\pcr)},-1) -- plot [smooth, domain=-1:1] ({sqrt(2*(\pt+1/2/(\pcr*\pcr+\pci*\pci)*(\pcr*\pwr*\pwr-\pcr*\pwi*\pwi+2*\pci*\pwr*\pwi))/\pcr+(\pcr*\pcr+\pci*\pci)/2/\pcr*\x*\x)},\x) -- ({sqrt(2*(\pt+1/2/(\pcr*\pcr+\pci*\pci)*(\pcr*\pwr*\pwr-\pcr*\pwi*\pwi+2*\pci*\pwr*\pwi))/\pcr+(\pcr*\pcr+\pci*\pci)/2/\pcr)},1);
\draw[Cyan] ({-sqrt(2*(\pt+1/2/(\pcr*\pcr+\pci*\pci)*(\pcr*\pwr*\pwr-\pcr*\pwi*\pwi+2*\pci*\pwr*\pwi))/\pcr+(\pcr*\pcr+\pci*\pci)/2/\pcr)},1) -- plot [smooth, domain=1:-1] ({-sqrt(2*(\pt+1/2/(\pcr*\pcr+\pci*\pci)*(\pcr*\pwr*\pwr-\pcr*\pwi*\pwi+2*\pci*\pwr*\pwi))/\pcr+(\pcr*\pcr+\pci*\pci)/2/\pcr*\x*\x)},\x) -- ({-sqrt(2*(\pt+1/2/(\pcr*\pcr+\pci*\pci)*(\pcr*\pwr*\pwr-\pcr*\pwi*\pwi+2*\pci*\pwr*\pwi))/\pcr+(\pcr*\pcr+\pci*\pci)/2/\pcr)},-1);
\fill[Cyan, opacity=0.5] ({sqrt(2*(\pt+1/2/(\pcr*\pcr+\pci*\pci)*(\pcr*\pwr*\pwr-\pcr*\pwi*\pwi+2*\pci*\pwr*\pwi))/\pcr+(\pcr*\pcr+\pci*\pci)/2/\pcr)},-1) -- plot [smooth, domain=-1:1] ({sqrt(2*(\pt+1/2/(\pcr*\pcr+\pci*\pci)*(\pcr*\pwr*\pwr-\pcr*\pwi*\pwi+2*\pci*\pwr*\pwi))/\pcr+(\pcr*\pcr+\pci*\pci)/2/\pcr*\x*\x)},\x) -- ({sqrt(2*(\pt+1/2/(\pcr*\pcr+\pci*\pci)*(\pcr*\pwr*\pwr-\pcr*\pwi*\pwi+2*\pci*\pwr*\pwi))/\pcr+(\pcr*\pcr+\pci*\pci)/2/\pcr)},1) -- ({-sqrt(2*(\pt+1/2/(\pcr*\pcr+\pci*\pci)*(\pcr*\pwr*\pwr-\pcr*\pwi*\pwi+2*\pci*\pwr*\pwi))/\pcr+(\pcr*\pcr+\pci*\pci)/2/\pcr)},1) -- plot [smooth, domain=1:-1] ({-sqrt(2*(\pt+1/2/(\pcr*\pcr+\pci*\pci)*(\pcr*\pwr*\pwr-\pcr*\pwi*\pwi+2*\pci*\pwr*\pwi))/\pcr+(\pcr*\pcr+\pci*\pci)/2/\pcr*\x*\x)},\x) -- ({-sqrt(2*(\pt+1/2/(\pcr*\pcr+\pci*\pci)*(\pcr*\pwr*\pwr-\pcr*\pwi*\pwi+2*\pci*\pwr*\pwi))/\pcr+(\pcr*\pcr+\pci*\pci)/2/\pcr)},-1) -- cycle;

\draw[Red,-] ({\pwr/\pcr},{(\pcr*\pwi-\pci*\pwr)/(\pcr*\pcr+\pci*\pci)})--({sqrt(2*(\pt+1/2/(\pcr*\pcr+\pci*\pci)*(\pcr*\pwr*\pwr-\pcr*\pwi*\pwi+2*\pci*\pwr*\pwi))/\pcr+(\pcr*\pcr+\pci*\pci)/2/\pcr)},{(\pcr*\pwi-\pci*\pwr)/(\pcr*\pcr+\pci*\pci)});
\draw[Red,-] ({\pwr/\pcr},{(\pcr*\pwi-\pci*\pwr)/(\pcr*\pcr+\pci*\pci)})--({\pwr/\pcr},1);
\fill[Red, opacity=0.3] ({\pwr/\pcr},{(\pcr*\pwi-\pci*\pwr)/(\pcr*\pcr+\pci*\pci)})--({sqrt(2*(\pt+1/2/(\pcr*\pcr+\pci*\pci)*(\pcr*\pwr*\pwr-\pcr*\pwi*\pwi+2*\pci*\pwr*\pwi))/\pcr+(\pcr*\pcr+\pci*\pci)/2/\pcr)},{(\pcr*\pwi-\pci*\pwr)/(\pcr*\pcr+\pci*\pci)})--({sqrt(2*(\pt+1/2/(\pcr*\pcr+\pci*\pci)*(\pcr*\pwr*\pwr-\pcr*\pwi*\pwi+2*\pci*\pwr*\pwi))/\pcr+(\pcr*\pcr+\pci*\pci)/2/\pcr)},1)--({\pwr/\pcr},1) --cycle;
\end{tikzpicture}
}
\end{picture}
\caption{The cases in which the intersection of hyperbolic region and sector has a compact connected component.}
\label{figureHyperbolaAndSectorAligned}
\end{figure}
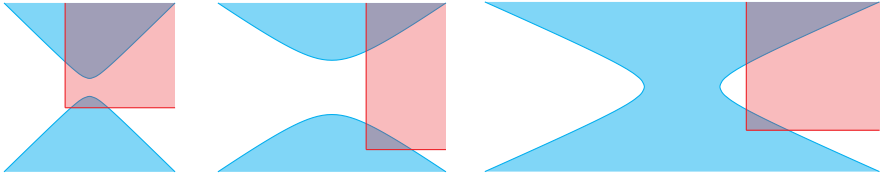

The cases of the sectors $\cS_2$, $\cS_3$ and $\cS_4$ are analogous. For the sectors $\cS_2$ and $\cS_3$, one needs to introduce the continuous functions $\philp{c},\philm{c}\colon \C_w\to\R$, which are given by

$$\philp{c}(\rew+i\imw)\defeq\begin{cases} 0 & \text{if $\rew< 0$} \\ \frac{\rew^2}{2\rec} &\text{if $\rew\geq 0$} \end{cases}$$
and
$$\philm{c}(\rew+i\imw)\defeq\begin{cases} -\frac{1}{2\rec|c|^2}(\imc\rew-\rec\imw)^2 = \dif{c}(w) & \text{if $\rew< 0$} \\ \frac{1}{2|c|^2}(\rec\rew^2-\rec\imw^2+2\imc\rew\imw)=\real{\frac{1}{2c}w^2} &\text{if $\rew\geq 0$} \end{cases}.$$

Set $\widehat{\cH}_-\defeq\{ w\in \C_w \mid \imc\rew-\rec\imw\geq 0 \}$ and $\widehat{\cH}_+\defeq \{w\in\C_w\mid\imc\rew-\rec\imw\leq 0\}$. Note that these half-planes only depend on $\arg C$.
The stalks suggest the following global statement. (Recall the notation from \cref{sectionEnhancedExponentials}.)

\begin{prop}\label{FourierOnSectors}
There are isomorphisms in $\Dbk{\C_w\times\R}$
\begin{align*}
\sF{\Exps{-\real{\frac{c}{2}}z^2}{\cS_1}{\C_z}}&\iso \ExpS{\phirp{c}}{\phirm{c}}{\widehat{\cH}_-}{\C_w}[1],
&\sF{\Exps{-\real{\frac{c}{2}}z^2}{\cS_2}{\C_z}}&\iso \ExpS{\philp{c}}{\philm{c}}{\widehat{\cH}_-}{\C_w}[1],\\
\sF{\Exps{-\real{\frac{c}{2}}z^2}{\cS_3}{\C_z}}&\iso \ExpS{\philp{c}}{\philm{c}}{\widehat{\cH}_+}{\C_w}[1],
&\sF{\Exps{-\real{\frac{c}{2}}z^2}{\cS_4}{\C_z}}&\iso \ExpS{\phirp{c}}{\phirm{c}}{\widehat{\cH}_+}{\C_w}[1].
\end{align*}
\end{prop}
\begin{proof}
We give a proof for the case of $\cS_1$. 

Set $A\defeq\{ (z,w,t)\in\C_z\times\C_w\times\R\mid t-\real{(zw+\frac{c}{2}z^2)} \geq 0\}\cap (\cS_1\times\C_w\times\R)$ and recall from \eqref{eq:enhFourierOfExpSheaf} that $\sF{\Exps{-\real{\frac{c}{2}}z^2}{\cS_1}{\C_z}}\iso \RR\qtild_!\kk_A[1]$.

First, consider the set
$$U\defeq\Big\{ (z,w,t)\in\C_z\times\C_w\times\R\,\Big|\, z\in \cS_1, t-\real{\Big(zw+\frac{c}{2}z^2\Big)} \geq 0, t<-\phirm{c}(w) \Big\}.$$
It is an open subset of $A$ and hence we have a distinguished triangle in $\Dbk{\C_w\times\R}$
$$\RR\qtild_!\kk_U\To \RR\qtild_!\kk_A\To \RR\qtild_!\kk_{A\setminus U}\ToPO.$$
By the projection formula, $\RR\qtild_!\kk_{A\setminus U}\iso \RR\qtild_!(\kk_{A}\otimes\qtild^{-1}\kk_{\{(w,t)\in\C_w\times\R\mid t\geq-\phirm{c}(w)\}})\iso \RR\qtild_!\kk_A\otimes \kk_{\{(w,t)\in\C_w\times\R\mid t\geq-\phirm{c}(w)\}}$ and hence it follows from \eqref{eq:CohomOfHyperbolaCapSectorDegNot-1} and \eqref{eq:CohomOfHyperbolaCapSectorDeg-1} that $\RR\qtild_!\kk_{A\setminus U}\iso 0$ and $\RR\qtild_!\kk_A\iso \RR\qtild_!\kk_U$.

Next, consider the set
\begin{align*}
B\defeq\Big\{ (z,w,t)\in\C_z\times\C_w\times\R\,\Big|\, &z_1=\frac{1}{\rec}\Big(\sqrt{2\rec t+\rew^2}-\rew\Big), z_2=0,\\
&\imc\rew-\rec\imw\geq 0, -\phirp{c}(w)\leq t<-\phirm{c}(w) \Big\}
\end{align*}
For fixed $\wf$ and $\tf$, the corresponding point $z=\frac{1}{\rec}\big(\sqrt{2\rec\tf+\wf_1^2}-\wf_1\big)$ is the rightmost intersection point of the hyperbolic region $\{\tf-\real{(z\wf+\frac{c}{2}z^2)}\geq 0\}$ with the horizontal border of the sector $\cS_1$. Moreover, $B$ is a closed subset of $U$ and we get a distinguished triangle in $\Dbk{\C_w\times\R}$
$$\RR\qtild_!\kk_{U\setminus B}\To \RR\qtild_!\kk_{U}\To \RR\qtild_!\kk_{B}\ToPO.$$
The stalks of the cohomology sheaves of $\RR\qtild_!\kk_{U\setminus B}$ are all trivial, and hence $\RR\qtild_!\kk_{B}\iso \RR\qtild_!\kk_{U}$.

Finally, one has $$\sF{\Exps{-\real{\frac{c}{2}}z^2}{\cS_1}{\C_z}}[-1]\iso \RR\qtild_!\kk_A\iso \RR\qtild_!\kk_B\iso \pitens{\widehat{\cH}_-} \kk_{\{ -\phirp{c}\leq t<\phirm{c} \}}\iso \ExpS{\phirp{c}}{\phirm{c}}{\widehat{\cH}_-}{\C_w}$$
since $\qtild$ induces a homeomorphism $B\TO{\approx} \{w\in\widehat{\cH}_-, -\phirp{c}(w)\leq t< -\phirm{c}(w)\}$.
\end{proof}

The computations of the enhanced Fourier--Sato transforms of $\Exps{-\real{\frac{c}{2}}z^2}{\cS_{k,k+1}}{\C_z}$ and $\Exps{-\real{\frac{c}{2}}z^2}{\{0\}}{\C_z}= \Exps{0\phantom{\varphi}}{\{0\}}{\C_z}$ are similar.

\begin{prop}
There are isomorphisms in $\Dbk{\C_w\times\R}$
\begin{align*}
\sF{\Exps{-\real{\frac{c}{2}}z^2}{\cS_{12}}{\C_z}}&\iso \ExpS{0}{\dif{c}}{\widehat{\cH}_-}{\C_w}[1],
&\sF{\Exps{-\real{\frac{c}{2}}z^2}{\cS_{23}}{\C_z}}&\iso \Exps{\philp{c}}{\C_w}{\C_w}[1],\\
\sF{\Exps{-\real{\frac{c}{2}}z^2}{\cS_{34}}{\C_z}}&\iso \ExpS{0}{\dif{c}}{\widehat{\cH}_+}{\C_w}[1],
&\sF{\Exps{-\real{\frac{c}{2}}z^2}{\cS_{41}}{\C_z}}&\iso \Exps{\phirp{c}}{\C_w}{\C_w}[1],
\end{align*}
$$\sF{\Exps{-\real{\frac{c}{2}}z^2}{\{0\}}{\C_z}}\iso \Exps{0\phantom{\varphi}}{\C_w}{\C_w}[1].$$
\end{prop}

Now that we computed the Fourier--Laplace transform of exponentials, let us briefly reflect on the impact of Fourier--Laplace on morphisms between those exponentials: Exponential enhanced sheaves are sheaves of the form $\kk_Z$ for some locally closed $Z\subseteq \C\times\R$. A morphism between two exponentials is therefore given by multiplication with an element $a\in\kk$ (at points where both stalks are $\kk$, it is multiplication by $a$). Since the enhanced Fourier--Sato transform consists only of tensor products and direct and inverse images along projections, one checks that the induced morphism between the enhanced Fourier--Laplace transforms of the exponentials is again given by multiplication with the same element $a\in\kk$.

\subsection{Enhanced Fourier--Sato transform of a Gaussian enhanced sheaf}\label{sectionFourierOfAlignedGaussian}
In this section, we will elaborate on the idea given at the end of \cref{MainStatementAligned} in order to describe the enhanced Fourier--Sato transform of $\cF_\sigma^{C,\theta_0,\ranks}$. We write for short $\cF\defeq \cF_\sigma^{C,\theta_0,\ranks}$.

To make notation easier, we will write $\Exps{\varphi}{Z}$ instead of $\Exps{\varphi}{Z}{X}$, and we shall assume $r_c=1$ for any $c\in C$. (One can replace any occurence of a direct sum
$\bigoplus_{c\in C} \Exps{\varphi_c}{Z}$ by $\bigoplus_{c\in C} \big(\Exps{\varphi_c}{Z}\big)^{r_c}$ and the word ``triangular'' by ``block-triangular'', and the proof is still valid.)\newline

Recall that we have defined a covering of the plane $\C_z$ by four closed sectors $\cS_k$, $k\in\Z/4\Z$. We set $\cH_+\defeq \cS_1\cup \cS_2$ and $\cH_-\defeq \cS_3\cup \cS_4$ as well as $\cS_{k,k+1}\defeq \cS_k\cap\cS_{k+1}$. On these sectors, we have isomorphisms
$$\alpha_k\colon \cF_{\cS_k}\TO{\sim}\expgzon{\cS_k}$$
and the gluing morphisms $\alpha_{k+1}\circ \alpha_k^{-1}$ on $\cS_{k,k+1}$ are given by the Stokes multipliers $\sigma_k$.

\subsubsection{Transform of restrictions to half-planes}
Let us start by investigating the short exact sequence in $\Modk{\C_z\times\R}$
\begin{equation}\label{eq:sequenceH+}
0\To \cF_{\cH_+}\To \cF_{\cS_1}\oplus\cF_{\cS_2}\To \cF_{\cS_{12}}\To 0.
\end{equation}
Via $\alpha_1$ and $\alpha_2$ (the latter used also for $\cF_{\cS_{12}}$), it is isomorphic to
$$0\To \cF_{\cH_+}\To \expgzon{\cS_1} \oplus \expgzon{\cS_2}\TO{\sigma_1-\1} \expgzon{\cS_{12}}\To 0.$$
Applying the enhanced Fourier--Sato transform and using the results of the previous section, we get a distinguished triangle in $\Dbk{\C_w\times\R}$
\begin{equation*}
\sF{(\cF_{\cH_+}})[-1]\To \bigoplus_{c\in C}\ExpS{\phirp{c}}{\phirm{c}}{\widehat{\cH}_-}\oplus \bigoplus_{c\in C} \ExpS{\philp{c}}{\philm{c}}{\widehat{\cH}_-} \TO{\sigma_1-\1} \bigoplus_{c\in C}\ExpS{0}{\dif{c}}{\widehat{\cH}_-}\ToPO.
\end{equation*}
Since the morphism $\sigma_1-\1$ is an epimorphism in $\Modk{\C_w\times\R}$, the associated long exact sequence yields the following proposition, comprising also the statements for $\cH_-$ and $L\defeq \cS_{41}\cup\cS_{23}$.

\begin{prop}\label{FourierOnHL}
Let $\cF=\cF_\sigma$ be an enhanced sheaf of pure Gaussian type (cf.\ \cref{defGaussianEnhancedSheaf}).
The complexes $\sF{(\cF_{\cH_+})}$, $\sF{(\cF_{\cH_-})}$ and $\sF{(\cF_{L})}$ are concentrated in degree $-1$. More precisely, there are isomorphisms in $\Dbk{\C_w\times\R}$
$$\sF{(\cF_{\cH_+})}[-1] \iso \ker\left( \sigma_1-\1\colon \bigoplus_{c\in C} \ExpS{\phirp{c}}{\phirm{c}}{\widehat{\cH}_-} \oplus \bigoplus_{c\in C} \ExpS{\philp{c}}{\philm{c}}{\widehat{\cH}_-} \to \bigoplus_{c\in C}\ExpS{0}{\dif{c}}{\widehat{\cH}_-} \right),$$

$$\sF{(\cF_{\cH_-})}[-1]\iso \ker\left( \1-\sigma_3\colon \bigoplus_{c\in C} \ExpS{\phirp{c}}{\phirm{c}}{\widehat{\cH}_+} \oplus \bigoplus_{c\in C}\ExpS{\philp{c}}{\philm{c}}{\widehat{\cH}_+} \to \bigoplus_{c\in C}\ExpS{0}{\dif{c}}{\widehat{\cH}_+}\right),$$

$$\sF{(\cF_L)}[-1]\iso \ker\left( \1-\sigma_4\sigma_3\colon \bigoplus_{c\in C} \Exps{\phirp{c}}{\C_w} \oplus \bigoplus_{c\in C} \Exps{\philp{c}}{\C_w} \to \bigoplus_{c\in C}\Exps{0\phantom{\varphi}}{\C_w} \right).$$
\end{prop}

\subsubsection{Transform on the whole plane}
We can now examine the sequence
\begin{equation}\label{eq:sequenceGlobal}
0\To \cF\To \cF_{\cH_+}\oplus\cF_{\cH_-}\To \cF_L\To 0,
\end{equation}
which will enable us to describe $\sF{\cF}$ and show that it is of the desired form on sectors.

Let us first make the morphism $\cF_{\cH_+}\oplus\cF_{\cH_-}\to \cF_L$ more explicit: Sequence \eqref{eq:sequenceH+} and similar sequences for $\cH_-$ and $L$ yield commutative diagrams\vspace{2ex}
\begin{equation}\label{eq:morphismH+L}
\begin{tikzcd}
0 \arrow{r}{} & \cF_{\cH_+} \arrow{r}{} \arrow{d}{} & \displaystyle\expgzon{\cS_1} \oplus \expgzon{\cS_2} \arrow{r}{\sigma_1-\1} \arrow{d}{\sigma_2}[swap]{\1} &[+10pt] \displaystyle\expgzon{\cS_{12}} \arrow{r}{} \arrow{d}{\sigma_1^{-1}} & 0\phantom{.}\\
0 \arrow{r}{} & \cF_{L} \arrow{r}{} & \displaystyle\expgzon{\cS_{41}}\oplus \expgzon{\cS_{23}} \arrow{r}{\1-\sigma_4\sigma_3} & \displaystyle\expgzon{\{0\}} \arrow{r}{} & 0
\end{tikzcd}
\end{equation}
and\vspace{2ex}
\begin{equation}\label{eq:morphismH-L}
\begin{tikzcd}
0 \arrow{r}{} & \cF_{\cH_-} \arrow{r}{} \arrow{d}{} & \displaystyle\expgzon{\cS_4} \oplus \expgzon{\cS_3} \arrow{r}{\1-\sigma_3} \arrow{d}{\1}[swap]{\sigma_4} &[+10pt] \displaystyle\expgzon{\cS_{34}} \arrow{r}{} \arrow{d}{\sigma_4} & 0\phantom{.}\\
0 \arrow{r}{} & \cF_{L} \arrow{r}{} & \displaystyle\expgzon{\cS_{41}}\oplus \expgzon{\cS_{23}} \arrow{r}{\1-\sigma_4\sigma_3} & \displaystyle\expgzon{\{0\}} \arrow{r}{} & 0.
\end{tikzcd}
\end{equation}

We would like to show that $\sF{\cF}$ is of pure Gaussian type $\widehat{C}=-1/C$. The considerations from the previous sections suggest using the following sectors:
\begin{align*}
\widehat{\cS}_1&\defeq \left\{ w\in \C_w\,\left| \, \arg w \in \left[-\pi+\arg C,-\frac{\pi}{2}\right]\text{if $w\neq 0$} \right\} \right., \\
\widehat{\cS}_2&\defeq \left\{ w\in \C_w\,\left| \, \arg w \in \left[-\frac{\pi}{2}, \arg C \right]\text{if $w\neq 0$} \right\} \right., \\
\widehat{\cS}_3&\defeq \left\{ w\in \C_w\,\left| \, \arg w \in \left[\arg C, \frac{\pi}{2} \right]\text{if $w\neq 0$} \right\} \right., \\
\widehat{\cS}_4&\defeq \left\{ w\in \C_w\,\left| \, \arg w \in \left[\frac{\pi}{2}, \pi+\arg C \right]\text{if $w\neq 0$} \right\} \right. .
\end{align*}
The Stokes directions for $\widehat{C}$ are $\frac{\pi}{4}+\frac{1}{2}\arg C+k\frac{\pi}{2}$. Hence, $\widehat{\theta}_0=\pi+\frac{1}{2}\arg C$ is indeed generic and the $\widehat{\cS}_k$ are compatible with the sectors $\widehat{S}_k=\{w\in \C\mid \arg w\in [\widehat{\theta}_0+(k-1)\frac{\pi}{2}, \widehat{\theta}_0+k\frac{\pi}{2}]\}$ in the sense of \cref{lemmaCompatibilitySectors}.
(An a posteriori justification for the choice of the generic direction is given by \cref{propStokesMatricesAfterFourier}.)
We have $\widehat{\cH}_+=\widehat{\cS}_3\cup \widehat{\cS}_4$ and $\widehat{\cH}_-=\widehat{\cS}_1\cup \widehat{\cS}_2$, and we set $\widehat{\cS}_{k,k+1}\defeq\widehat{\cS}_k\cap\widehat{\cS}_{k+1}$.

\begin{prop}\label{propFourierTrivializationOnSectors}
The enhanced Fourier--Sato transform $\sF{\cF}$ is concentrated in degree zero and for every $k\in\Z/4\Z$, we have an isomorphism in $\Dbk{\C_w\times\R}$
$$(\sF{\cF})_{\widehat{\cS}_k}\iso \expgwon{\widehat{\cS}_k}.$$
In particular, $\sF{\cF}$ is of pure Gaussian type $\widehat{C}=-1/C$.
\end{prop}
\begin{proof}
We prove the desired isomorphism for $k=1$.

From \eqref{eq:sequenceGlobal}, we get a distinguished triangle
\begin{equation}\label{eq:triangleGlobalExplicit}
\ker(\sigma_1-\1)_{\widehat{\cS}_1}\oplus \ker(\1-\sigma_3)_{\widehat{\cS}_1}\TO{(\1|\sigma_2)-(\sigma_4|\1)}\ker(\1-\sigma_4\sigma_3)_{\widehat{\cS}_1} \To (\sF{\cF})_{\widehat{\cS}_1} \ToPO.
\end{equation}
Here, the kernels are the ones from Proposition \ref{FourierOnHL}. The first morphism is induced by the ones described in \eqref{eq:morphismH+L} and \eqref{eq:morphismH-L}.

Firstly, we note that $\ker(\1-\sigma_3)_{\widehat{\cS}_1}\iso 0$ since it is the kernel of the morphism
$$\bigoplus_{c\in C} \ExpS{\phirp{c}}{\phirm{c}}{\widehat{\cS}_{41}} \oplus \bigoplus_{c\in C} \ExpS{\philp{c}}{\philm{c}}{\widehat{\cS}_{41}} \TO{\1-\sigma_3} \bigoplus_{c\in C}\ExpS{0}{\dif{c}}{\widehat{\cS}_{41}}$$
and on $\widehat{\cS}_{41}$ we have $\imc\rew-\rec\imw=0$, hence $\phirp{c}(w)=\phirm{c}(w)$ and $\philp{c}(w)=\philm{c}(w)$.

Secondly, we determine $\ker(\sigma_1-\1)_{\widehat{\cS}_1}$: It is the first object in the short exact sequence
\begin{align}\label{eq:sequenceSigma1}
0\To\bigoplus_{c\in C} \ExpS{\frac{\rew^2}{2\rec}}{\real{\frac{1}{2c}w^2}}{\widehat{\cS}_1} \TO{(\1,\sigma_1)} \bigoplus_{c\in C} \ExpS{\frac{\rew^2}{2\rec}}{\real{\frac{1}{2c}w^2}}{\widehat{\cS}_1}&\oplus \bigoplus_{c\in C} \ExpS{0}{\dif{c}}{\widehat{\cS}_1} \\ &\TO{\sigma_1-\1} \bigoplus_{c\in C}\ExpS{0}{\dif{c}}{\widehat{\cS}_1}\To 0.\notag
\end{align}

Thirdly, we find $\ker(\1-\sigma_4\sigma_3)_{\widehat{\cS}_1}$ as the first object in the short exact sequence
\begin{equation}\label{eq:sequenceSigma43}
0\To \bigoplus_{c\in C}\Exps{\frac{\rew^2}{2\rec}}{\widehat{\cS}_1} \TO{(\1,\sigma_2\sigma_1)} \bigoplus_{c\in C}\Exps{\frac{\rew^2}{2\rec}}{\widehat{\cS}_1} \oplus \bigoplus_{c\in C}\Exps{0\phantom{\varphi}}{\widehat{\cS}_1} \TO{\1-\sigma_4\sigma_3} \bigoplus_{c\in C}\Exps{0\phantom{\varphi}}{\widehat{\cS}_1} \To 0.
\end{equation}

Finally, there is a commutative diagram in which the sequences \eqref{eq:sequenceSigma1} and \eqref{eq:sequenceSigma43} appear as the columns, and which has exact rows and columns:\vspace{2ex}
\begin{equation}\label{eq:diagramOnS1}
\begin{tikzcd}
&[-30pt]0\arrow{d}{} &[+10pt] 0\arrow{d}{}&[-20pt]&[-10pt]\\
0\arrow{r}{}&\displaystyle\bigoplus_{c\in C}\ExpS{\frac{\rew^2}{2\rec}}{\real{\frac{1}{2c}w^2}}{\widehat{\cS}_1} \arrow{d}{(\1,\sigma_1)}\arrow{r}{\1} &\displaystyle\bigoplus_{c\in C}\Exps{\frac{\rew^2}{2\rec}}{\widehat{\cS}_1} \arrow{d}{(\1,\sigma_2\sigma_1)}\arrow{r}{\1} & \displaystyle\bigoplus_{c\in C}\Exps{\real{\frac{1}{2c}w^2}}{\widehat{\cS}_1}\arrow{r}{}&0\\[+10pt]
&\displaystyle\bigoplus_{c\in C} \ExpS{\frac{\rew^2}{2\rec}}{\real{\frac{1}{2c}w^2}}{\widehat{\cS}_1} \oplus \bigoplus_{c\in C}\ExpS{0}{\dif{c}}{\widehat{\cS}_1}  \arrow{d}{\sigma_1-\1}\arrow{r}{\1|\sigma_2} & \displaystyle\bigoplus_{c\in C}\Exps{\frac{\rew^2}{2\rec}}{\widehat{\cS}_1} \oplus \bigoplus_{c\in C}\Exps{0\phantom{\varphi}}{\widehat{\cS}_1}  \arrow{d}{\1-\sigma_4\sigma_3}
\\[+10pt]
&\displaystyle\bigoplus_{c\in C}\ExpS{0}{\dif{c}}{\widehat{\cS}_1}\arrow{d}{}\arrow{r}{\sigma_1^{-1}} &\displaystyle \bigoplus_{c\in C}\Exps{0\phantom{\varphi}}{\widehat{\cS}_1}\arrow{d}{}
\\
&0 & 0
\end{tikzcd}
\end{equation}
Comparing the upper row of this diagram with the long exact sequence associated to \eqref{eq:triangleGlobalExplicit}, the statement of the proposition follows.
\end{proof}

\subsection{Stokes multipliers of the Fourier--Laplace transform}
We have seen in \cref{propFourierTrivializationOnSectors} that $\sF{\cF}$ is isomorphic to a direct sum of exponential enhanced sheaves on each of the $\widehat{\cS}_k$ (and such isomorphisms have actually been constructed). Therefore, on each of the half-lines $\widehat{\cS}_{k,k+1}$ we have two trivializing isomorphisms $\widehat\alpha_k$ and $\widehat{\alpha}_{k+1}$ coming from the ones on the two adjacent sectors. Our aim is to find matrices $\widehat\sigma_k$ representing an automorphism of $\expgwon{\widehat{\cS}_{k,k+1}}$ such that the following diagram commutes for any $k\in\Z/4\Z$:
\begin{center}
\begin{tikzcd}
\displaystyle\expgwon{\widehat{\cS}_{k,k+1}}\arrow{rr}{\widehat{\sigma}_k}&&\displaystyle\expgwon{\widehat{\cS}_{k,k+1}}\\ \\
&(\sF{\cF})_{\widehat{\cS}_{k,k+1}}\arrow{luu}{\widehat{\alpha}_k}[swap]{\iso}\arrow{ruu}{\iso}[swap]{\widehat{\alpha}_{k+1}}
\end{tikzcd}
\end{center}
Note that $\expgwon{\widehat{\cS}_{k,k+1}}=\bigoplus_{\widehat{c}\in\widehat{C}}\Exps{-\real{\frac{\widehat{c}}{2}w^2}}{\widehat{\cS}_{k,k+1}}$ and the order on $\widehat{C}$ with respect to $\widehat{\theta}_0$ is the one induced by the order on $C$ with respect to $\theta_0$, i.e.\ $c<_{\theta_0} d$ if and only if $\widehat{c}<_{\widehat{\theta}_0}\widehat{d}$.
\begin{prop}\label{propStokesMatricesAfterFourier}
Gluing matrices for $\sF{\cF}$ are given by $\widehat\sigma_k=\sigma_k$, $k\in\Z/4\Z$.
\end{prop}
\begin{proof}
Let us give the proof for $\widehat{\sigma}_1=\sigma_1$.

By what we learnt in \cref{propFourierTrivializationOnSectors}, the triangle \eqref{eq:triangleGlobalExplicit} is actually a short exact sequence (identifying $\sF{\cF}$ with $\mathrm{H}^0(\sF{\cF})$)
\begin{equation*}
0\To \ker(\sigma_1-\1)\oplus\ker(\1-\sigma_3)\TO{(\1|\sigma_2)-(\sigma_4|\1)}\ker(\1-\sigma_4\sigma_3)\To \sF{\cF}\To 0.
\end{equation*}
On $\widehat{\cS}_1$ (i.e.\ applying $(\bullet)_{\widehat{\cS}_1}$), it induces
\begin{equation*}
0\To \ker(\sigma_1-\1)_{\widehat{\cS}_1}\TO{\1|\sigma_2}\ker(\1-\sigma_4\sigma_3)_{\widehat{\cS}_1}\To (\sF{\cF})_{\widehat{\cS}_1}\To 0
\end{equation*}
since we proved $\ker(\1-\sigma_3)_{\widehat{\cS}_1}\iso 0$. We obtained determinations of $\ker(\sigma_1-\1)_{\widehat{\cS}_1}$ and $\ker(\1-\sigma_4\sigma_3)_{\widehat{\cS}_1}$ and hence the isomorphism $\widehat{\alpha}_1$ as the third vertical arrow in the diagram\vspace{2ex}
\begin{equation}\label{eq:oddDiagramS1}
\begin{tikzcd}
0\arrow{r}& \ker(\sigma_1-\1)_{\widehat{\cS}_1} \arrow{r}{\1|\sigma_2}\arrow{d}{\iso}& \ker(\1-\sigma_4\sigma_3)_{\widehat{\cS}_1} \arrow{r}{} \arrow[Red]{d}{\iso} & (\sF{\cF})_{\widehat{\cS}_1} \arrow[blue]{d}{\iso} \arrow{r} & 0\\
0\arrow{r} &  \displaystyle\bigoplus_{c\in C} \ExpS{\frac{\rew^2}{2\rec}}{\real{\frac{1}{2c}w^2}}{\widehat{\cS}_1} \arrow{r}{\1} & \displaystyle\bigoplus_{c\in C}\Exps{\frac{\rew^2}{2\rec}}{\widehat{\cS}_1} \arrow{r}{\1} & \displaystyle\bigoplus_{c\in C}\Exps{\real{\frac{1}{2c}w^2}}{\widehat{\cS}_1}\arrow{r} & 0
\end{tikzcd}
\end{equation}
Similarly, $\widehat{\alpha}_2$ is obtained from the diagram\vspace{2ex}
\begin{equation}\label{eq:oddDiagramS2}
\begin{tikzcd}
0\arrow{r}& \ker(\sigma_1-\1)_{\widehat{\cS}_2} \arrow{r}{\1|\sigma_2}\arrow[dashed]{d}{\iso}& \ker(\1-\sigma_4\sigma_3)_{\widehat{\cS}_2} \arrow{r}{} \arrow[Red,dashed]{d}{\iso} & (\sF{\cF})_{\widehat{\cS}_2} \arrow[dashed,blue]{d}{\iso} \arrow{r} & 0\\
0\arrow{r} &  \displaystyle\bigoplus_{c\in C} \ExpS{\frac{\rew^2}{2\rec}}{\real{\frac{1}{2c}w^2}}{\widehat{\cS}_2}  \arrow{r}{\sigma_2} & \displaystyle\bigoplus_{c\in C}\Exps{\frac{\rew^2}{2\rec}}{\widehat{\cS}_2} \arrow{r}{\sigma_2^{-1}} & \displaystyle\bigoplus_{c\in C}\Exps{\real{\frac{1}{2c}w^2}}{\widehat{\cS}_2}\arrow{r} & 0
\end{tikzcd}
\end{equation}
Now we can take the right square of diagrams \eqref{eq:oddDiagramS1} and \eqref{eq:oddDiagramS2}, apply the functor $(\bullet)_{\widehat{\cS}_{12}}$ and identify their first lines, and we obtain
\begin{equation}\label{eq:diagramSigmaHat}
\begin{tikzcd}
\displaystyle\bigoplus_{c\in C}\Exps{\frac{\rew^2}{2\rec}}{\widehat{\cS}_{12}}\arrow{r}{\1} \arrow[bend left,orange]{dd} & \displaystyle\bigoplus_{c\in C}\Exps{\real{\frac{1}{2c}w^2}}{\widehat{\cS}_{12}} \arrow[bend left,Purple]{dd}{\widehat{\sigma}_1}\\
\ker(\1-\sigma_4\sigma_3)_{\widehat{\cS}_{12}} \arrow{r}{} \arrow[Red,dashed,swap]{d}{\iso} \arrow[Red]{u}{\iso}& (\sF{\cF})_{\widehat{\cS}_{12}} \arrow[blue,dashed,swap]{d}{\iso} \arrow[blue]{u}{\iso}\\
\displaystyle\bigoplus_{c\in C}\Exps{\frac{\rew^2}{2\rec}}{\widehat{\cS}_{12}} \arrow{r}{\sigma_2^{-1}} & \displaystyle\bigoplus_{c\in C}\Exps{\real{\frac{1}{2c}w^2}}{\widehat{\cS}_{12}}
\end{tikzcd}
\end{equation}
and the purple arrow is the one in question. Therefore, it remains to determine the orange one.

The object $\ker(\1-\sigma_4\sigma_3)_{\widehat{\cS}_1}$ was determined by the short exact sequence
\begin{equation*}
0\To \bigoplus_{c\in C}\Exps{\frac{\rew^2}{2\rec}}{\widehat{\cS}_1} \TO{(\1,\sigma_2\sigma_1)} \bigoplus_{c\in C}\Exps{\frac{\rew^2}{2\rec}}{\widehat{\cS}_1} \oplus \bigoplus_{c\in C}\Exps{0\phantom{\varphi}}{\widehat{\cS}_1}\TO{\1-\sigma_4\sigma_3} \bigoplus_{c\in C}\Exps{0\phantom{\varphi}}{\widehat{\cS}_1}\To 0
\end{equation*}
and the object $\ker(\1-\sigma_4\sigma_3)_{\widehat{\cS}_2}$ by the sequence
$$0\To \bigoplus_{c\in C}\Exps{\frac{\rew^2}{2\rec}}{\widehat{\cS}_2} \TO{(\sigma_1^{-1}\sigma_2^{-1},\1)}  \bigoplus_{c\in C}\Exps{0\phantom{\varphi}}{\widehat{\cS}_2}\oplus \bigoplus_{c\in C}\Exps{\frac{\rew^2}{2\rec}}{\widehat{\cS}_2} \TO{\1-\sigma_4\sigma_3} \displaystyle\bigoplus_{c\in C}\Exps{0\phantom{\varphi}}{\widehat{\cS}_2}\To 0.$$
Applying the functor $(\bullet)_{\widehat{\cS}_{12}}$, the second and third objects of both sequences are identified (since $\rew=0$ on $\widehat{\cS}_{12}$) and the induced isomorphism between the first objects (which is the orange arrow from \eqref{eq:diagramSigmaHat}) is clearly given by $\sigma_2\sigma_1$. Therefore, it follows from \eqref{eq:diagramSigmaHat} that $\widehat{\sigma}_1=\sigma_1$.
\end{proof}

This concludes the proof of \cref{thmAlignedParametersEnhancedSheaves}.

\section{A more general case}
In this section, we show how the methods of the previous section can be adapted to a case with weaker assumptions on the parameter set $C$. In contrast to \cite{Sab16}, this yields an explicit solution to the problem of finding a transformation rule for Stokes data in more general cases than in \cref{sectionAligned}. Although Corollary 4.19 in loc.\ cit.\ provided a theoretical answer by stating that arbitrary parameter configurations can be deformed into those studied in the previous section, this answer was not at all explicit.

We restrict to the case where $C=\{c,d\}$ consists of two parameters and the ranks of the regular parts are $r_c=r_d=1$ (and we suppress $\ranks$ in our notation).

\begin{cond}
We say that an ordered pair $(c,d)$ of nonzero complex numbers $c,d\in\C^\times$ satisfies condition \eqref{eq:ConditionsOnC} if the following is satisfied:
\begin{equation}\label{eq:ConditionsOnC}
\rec>0,\quad \imc\geq 0,\quad \red>\rec,\quad \text{and}\quad \frac{\imd}{\red}\geq \frac{\imc}{\rec}.  \tag{\katare}
\end{equation}
where we write $c=\rec+i\imc$ and $d=\red+i\imd$ with their real and imaginary parts.
\end{cond}
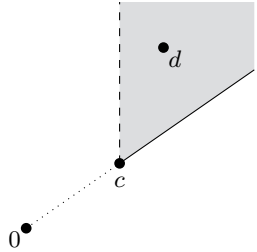
\begin{figure}[ht]
\centering
\begin{tikzpicture}
\fill[Gray!30!white] ({1.5*cos(35)},{1.5*sin(35)}) -- (3,{3*tan(35)}) -- (3,3) -- ({1.5*cos(35)},3) -- cycle;
\draw[dotted] (0,0) -- ({1.5*cos(35)},{1.5*sin(35)});
\draw ({1.5*cos(35)},{1.5*sin(35)}) -- (3,{3*tan(35)});
\draw[dashed] ({1.5*cos(35)},{1.5*sin(35)}) -- ({1.5*cos(35)},3);
\fill ({1.5*cos(35)},{1.5*sin(35)}) circle (2pt);
\fill (0,0) circle (2pt);
\fill ({3*cos(53)},{3*sin(53)}) circle (2pt);
\node at (-0.15,-0.15) {\small $0$};
\node at ({1.5*cos(35)},{1.5*sin(35)-0.25}) {\small $c$};
\node at ({3*cos(53)+0.15},{3*sin(53)-0.15}) {\small $d$};
\end{tikzpicture}
\caption{Let $\arg c\in[0,\frac{\pi}{2})$. The pair $(c,d)$ satisfies condition \eqref{eq:ConditionsOnC} if and only if $d$ lies in the cone with vertex $c$ and bounded by the directions $\arg c$ (included) and $\frac{\pi}{2}$ (excluded).}
\end{figure}

Let $C=\{c,d\}\subset \C^\times$ such that \eqref{eq:ConditionsOnC} is satisfied. Set $\theta_0\defeq -\frac{1}{2}\arg c$. It is a generic direction since the Stokes directions are $-\frac{\pi}{4}-\frac{\arg(d-c)}{2}+k\frac{\pi}{2}$. Let $\sigma=(\sigma_k)_{k\in\Z/4\Z}$ be a family of four $2\times 2$-matrices such that $\sigma_k$ is upper-triangular (resp.\ lower-triangular) if $k$ is odd (resp.\ even) and $\sigma_4\sigma_3\sigma_2\sigma_1=\1$.
\begin{thm}\label{thmNotAligned}
Let $C$, $\theta_0$ and $\sigma$ be as above. If we set $\widehat{C}\defeq \{-\frac{1}{c},-\frac{1}{d}\}$ and $\widehat{\theta}_0\defeq \pi-\theta_0$, there is an isomorphism in $\DD^\mathrm{b}(\kk_{\C\times\R})$
$$\sF{\cF^{C,\theta_0}_\sigma}\iso \cF^{\widehat{C},\widehat{\theta}_0}_\sigma.$$
\end{thm}
\begin{cor}
Let $\cM$ be of pure Gaussian type $C=\{c,d\}$ such that \eqref{eq:ConditionsOnC} holds. Then $\dF{\cM}$ is of pure Gaussian type $\widehat{C}=\{-\frac{1}{c},-\frac{1}{d}\}$. Moreover, if $\sigma=(\sigma_k)_{k\in \Z/4\Z}$ is a family of Stokes multipliers for $\cM$ with respect to the generic direction $\theta_0=-\frac{1}{2}\arg c$, then $\sigma$ is also a family of Stokes multipliers for $\dF{\cM}$ with respect to the generic direction $\widehat{\theta}_0=\pi-\theta_0$.
\end{cor}

\begin{proof}[Proof of \cref{thmNotAligned}]
First, we choose a sector decomposition analogously to \cref{sectionFourierOfExponentialsOnSectors}, replacing $\arg C$ by $\arg c$, i.e.\
$$\cS_1\defeq \left\{ z\in \C_z\,\left| \, \arg z \in \left[0,\frac{\pi}{2}-\arg c\right]\text{if $z\neq 0$} \right\} \right.\quad \text{etc.}$$

Next, we compute the enhanced Fourier--Sato transforms of the exponentials involved: For the parameter $c$, this is exactly the same computation that we performed above, i.e.\
$$\sF{\Exps{-\real{\frac{c}{2}}z^2}{\cS_1}{\C_z}}\iso \ExpS{\phirp}{\phirm}{\widehat{\cH}_-}{\C_w}[1]\quad \text{etc.}$$
(see \cref{FourierOnSectors}, we write $\phirp$ instead of $\phirp{c}$ etc. here).
For the exponentials $\Exps{-\real{\frac{d}{2}z^2}}{\cS_k}{\C_z}$, one proceeds similarly. However, the coordinate transform for the parameter $d$ (similar to \eqref{eq:coordinateTransform}) does not transform $\cS_k$ into right-angled sectors. Hence, the geometry of the intersection spaces is more involved, yet it is still not too difficult to determine the compactly supported cohomologies, and we find that
\begin{align*}
\sF{\Exps{-\real{\frac{d}{2}z^2}}{\cS_1}{\C_z}}&\iso \ExpS{\psirp}{\psirm}{Y_1}{\C_w}[1],
&\sF{\Exps{-\real{\frac{d}{2}z^2}}{\cS_2}{\C_z}}&\iso \ExpS{\psilp}{\psilm}{Y_2}{\C_w}[1],\\
\sF{\Exps{-\real{\frac{d}{2}z^2}}{\cS_3}{\C_z}}&\iso \ExpS{\psilp}{\psilm}{Y_3}{\C_w}[1],
&\sF{\Exps{-\real{\frac{d}{2}z^2}}{\cS_4}{\C_z}}&\iso \ExpS{\psirp}{\psirm}{Y_4}{\C_w}[1];\\
\sF{\Exps{-\real{\frac{d}{2}z^2}}{\cS_{12}}{\C_z}}&\iso \ExpS{0}{\zeta}{\widehat{\cH}_-}{\C_w}[1],
&\sF{\Exps{-\real{\frac{d}{2}z^2}}{\cS_{23}}{\C_z}}&\iso \Exps{\psilp}{\C_w}{\C_w}[1],\\
\sF{\Exps{-\real{\frac{d}{2}z^2}}{\cS_{34}}{\C_z}}&\iso \ExpS{0}{\zeta}{\widehat{\cH}_+}{\C_w}[1],
&\sF{\Exps{-\real{\frac{d}{2}z^2}}{\cS_{41}}{\C_z}}&\iso \Exps{\psirp}{\C_w}{\C_w}[1].
\end{align*}
Here, the functions $\psirp,\psirm\colon \C_w\to\R$ are defined by
\begin{align*}
\psirp(w)\defeq \begin{cases}\frac{\rew^2}{2\red}& \text{if $\rew\leq 0$}\\ 0 & \text{if $\rew>0$} \end{cases}
\end{align*}
and
\begin{align*}
\psirm(w)\defeq \begin{cases}\real{\frac{1}{2d}w^2}& \text{if $(\rec \imd - \imc \red)\imw\leq-(\rec\red+\imc\imd)\rew$}\\\zeta(w) & \text{if $(\rec \imd - \imc \red)\imw>-(\rec\red+\imc\imd)\rew$} \end{cases},
\end{align*}
where $\zeta(w)\defeq -\frac{(\imc \rew-\rec\imw)^2}{2(\rec^2\red+2\rec\imc\imd-\imc^2\red)}$ and $\psilp,\psilm\colon \C_w\to \R$ are similar (with cases interchanged). Moreover,
\begin{align*}
Y_1&\defeq \left\{ w\in  \C_w \,\left|\, \imw\leq \min\!\Big(\frac{\imc}{\rec}\rew,\frac{\imd}{\red}\rew\Big) \right. \right\},\\
Y_2&\defeq \left\{ w\in  \C_w \,\left|\, \imw\leq \max\!\Big(\frac{\imc}{\rec}\rew,\frac{\imd}{\red}\rew\Big) \right. \right\},\\
Y_3&\defeq \left\{ w\in  \C_w \,\left|\, \imw\geq \max\!\Big(\frac{\imc}{\rec}\rew,\frac{\imd}{\red}\rew\Big) \right. \right\},\\
Y_4&\defeq \left\{ w\in  \C_w \,\left|\, \imw\geq \min\!\Big(\frac{\imc}{\rec}\rew,\frac{\imd}{\red}\rew\Big) \right. \right\}.
\end{align*}
\begin{figure}[ht]
\centering
\begin{picture}(360,90)
\put(0,10){
\begin{tikzpicture}[scale=0.7]
\fill[SeaGreen!70!white] (0,0) -- (2,{2/3}) -- (2,-2) -- (-2,-2) -- ({-4/3},-2) -- cycle;
\draw[thick] (0,0) -- (2,{2/3});
\draw[thick] (0,0) -- ({-4/3},-2);
\draw (2,2) -- (-2,2) -- (-2,-2) -- (2,-2) -- cycle;
\draw[dashed,MidnightBlue,thick] (0,-2) -- (0,2);
\draw[dashed,MidnightBlue,thick] (-2,{-2/3}) -- (2,{2/3});
\node at (-1.5,1.5) {\textcolor{MidnightBlue}{$\widehat{\cS}_4$}};
\node at (1.5,-1.5) {\textcolor{MidnightBlue}{$\widehat{\cS}_2$}};
\node at (1.5,1.5) {\textcolor{MidnightBlue}{$\widehat{\cS}_3$}};
\node at (-1.5,-1.5) {\textcolor{MidnightBlue}{$\widehat{\cS}_1$}};
\end{tikzpicture}
}
\put(35,0){
$Y_1$
}
\put(90,10){
\begin{tikzpicture}[scale=0.7]
\fill[SeaGreen!70!white] (0,0) -- (-2,{-2/3}) -- (-2,-2) -- (2,-2) -- (2,2) -- ({4/3},2) -- cycle;
\draw[thick] (0,0) -- (-2,{-2/3});
\draw[thick] (0,0) -- ({4/3},2);
\draw (2,2) -- (-2,2) -- (-2,-2) -- (2,-2) -- cycle;
\draw[dashed,MidnightBlue,thick] (0,-2) -- (0,2);
\draw[dashed,MidnightBlue,thick] (-2,{-2/3}) -- (2,{2/3});
\end{tikzpicture}
}
\put(125,0){
$Y_2$
}
\put(180,10){
\begin{tikzpicture}[scale=0.7]
\fill[SeaGreen!70!white] (0,0) -- (-2,{-2/3}) -- (-2,2) -- ({4/3},2) -- cycle;
\draw[thick] (0,0) -- (-2,{-2/3});
\draw[thick] (0,0) -- ({4/3},2);
\draw (2,2) -- (-2,2) -- (-2,-2) -- (2,-2) -- cycle;
\draw[dashed,MidnightBlue,thick] (0,-2) -- (0,2);
\draw[dashed,MidnightBlue,thick] (-2,{-2/3}) -- (2,{2/3});
\end{tikzpicture}
}
\put(215,0){
$Y_3$
}
\put(270,10){
\begin{tikzpicture}[scale=0.7]
\fill[SeaGreen!70!white] (0,0) -- (2,{2/3}) -- (2,2) -- (-2,2) -- (-2,-2) -- ({-4/3},-2) -- cycle;
\draw[thick] (0,0) -- (2,{2/3});
\draw[thick] (0,0) -- ({-4/3},-2);
\draw (2,2) -- (-2,2) -- (-2,-2) -- (2,-2) -- cycle;
\draw[dashed,MidnightBlue,thick] (0,-2) -- (0,2);
\draw[dashed,MidnightBlue,thick] (-2,{-2/3}) -- (2,{2/3});
\end{tikzpicture}
}
\put(305,0){
$Y_4$
}
\end{picture}
\caption{The sets $Y_k$ and their relative positions with respect to the sectors $\widehat{\cS}_k$.}
\label{figureY}
\end{figure}

One can now determine $\sF{(\cF_{\cH_+})}$, $\sF{(\cF_{\cH_-})}$ and $\sF{(\cF_{L})}$ by enhanced Fourier--Sato transform of short exact sequences (cf.\ \cref{FourierOnHL}). One then proves isomorphisms
$$(\sF{\cF})_{\widehat{\cS}_k}\iso \Exps{\real{\frac{1}{2c}w^2}}{\widehat{\cS}_k}{\C_w}\oplus \Exps{\real{\frac{1}{2d}w^2}}{\widehat{\cS}_k}{\C_w},$$
where the sectors $\widehat{\cS}_k$ are defined as in \cref{sectionFourierOfAlignedGaussian} (with $\arg C$ replaced by $\arg c$), i.e.\
$$\widehat{\cS}_1\defeq \left\{ w\in \C_w\,\left| \, \arg w \in \left[-\pi+\arg c,-\frac{\pi}{2}\right]\text{if $w\neq 0$} \right\} \right. \quad \text{etc.}$$
The main difference is the fact that the supports of the Fourier--Sato transforms of the $\Exps{-\real{\frac{d}{2}z^2}}{\cS_k}{\C_z}$ (i.e.\ the sets $Y_k$) are not unions of these sectors (see \cref{figureY}). Therefore, if we want to mimick the proof of \cref{propFourierTrivializationOnSectors} (for $k=1$), we will not have $\ker(\1-\sigma_3)\iso 0$, but the second summand of $\ker(\sigma_1-\1)$ ``splits'' into two parts. The diagram correponding to \eqref{eq:diagramOnS1} in this case then looks as follows (we write direct sums vertically):
\begin{center}\small
\begin{tikzcd}
&[-30pt] 0\arrow{d}&[+30pt] 0\arrow{d}&[-20pt] &[-10pt]\\
0\arrow{r} & \vsum{\ExpS{\frac{\rew^2}{2\rec}}{\real{\frac{1}{2c}w^2}}{\widehat{\cS}_1}}{\ExpS{\frac{\rew^2}{2\red}}{\real{\frac{1}{2d}w^2}}{Y_1\cap \widehat{\cS}_1}} \oplus \vsum{\phantom{\Exps{\frac{\rew^2}{2\rec}}{\widehat{\cS}_1}}0\phantom{\Exps{\frac{\rew^2}{2\rec}}{\widehat{\cS}_1}}}{\ExpS{\frac{\rew^2}{2\red}}{\real{\frac{1}{2d}w^2}}{Y_4\cap\widehat{\cS}_1}}\arrow{d}{-\sigma_4^{-1}}[swap]{(\1,\sigma_1)} \arrow{r}{\1+\1} & \vsum{\Exps{\frac{\rew^2}{2\rec}}{\widehat{\cS}_1}}{\Exps{\frac{\rew^2}{2\red}}{\widehat{\cS}_1}}\arrow{d}{(\1,\sigma_3^{-1}\sigma_4^{-1})} \arrow{r}{\1} & \vsum{\Exps{\real{\frac{1}{2c}w^2}}{\widehat{\cS}_1}}{\Exps{\real{\frac{1}{2d}w^2}}{\widehat{\cS}_1}}\arrow{r} & 0\\
&\left(\vsum{\ExpS{\frac{\rew^2}{2\rec}}{\real{\frac{1}{2c}w^2}}{\widehat{\cS}_1}}{\ExpS{\frac{\rew^2}{2\red}}{\real{\frac{1}{2d}w^2}}{Y_1\cap \widehat{\cS}_1}}\oplus \vsum{\ExpS{0}{\eta}{\widehat{\cS}_1}}{\ExpS{0}{\zeta}{\widehat{\cS}_1}}\right)\oplus \vsum{0}{\ExpS{\frac{\rew^2}{2\red}}{\real{\frac{1}{2d}w^2}}{Y_4\cap\widehat{\cS}_1}}\arrow{d}{0}[swap]{\sigma_1-\1} \arrow{r}{(\1|\sigma_2)-(\sigma_4,0)} & \vsum{\Exps{\frac{\rew^2}{2\rec}}{\widehat{\cS}_1}}{\Exps{\frac{\rew^2}{2\red}}{\widehat{\cS}_1}}\oplus \vsum{\Exps{0\phantom{\varphi}}{\widehat{\cS}_1}}{\Exps{0\phantom{\varphi}}{\widehat{\cS}_1}}\arrow{d}{\1-\sigma_4\sigma_3}\\ &\vsum{\ExpS{0}{\eta}{\widehat{\cS}_1}}{\ExpS{0}{\zeta}{\widehat{\cS}_1}}\oplus \vsum{0}{0}\arrow{d} \arrow{r}{\sigma_1^{-1}+0} & \vsum{\Exps{0\phantom{\varphi}}{\widehat{\cS}_1}}{\Exps{0\phantom{\varphi}}{\widehat{\cS}_1}}\arrow{d}\\
& 0 & 0
\end{tikzcd}\normalsize
\end{center}
Although the left part of the diagram becomes more complicated, the cokernel of the morphism in the first line is as desired.

The computation of transition matrices for $\sF{\cF}$ then works analogously to that in the aligned case.
\end{proof}

A generalization to more than two parameters (and ranks not equal to $1$) is easily possible:
One then needs to require that the elements of $C=\{\pc{1},\ldots, \pc{n}\}$ satisfy condition \eqref{eq:ConditionsOnC} ``pairwise'', i.e.\ $(\pc{k},\pc{k+1})$ satisfies condition \eqref{eq:ConditionsOnC} for any $k\in\{1,\ldots, n-1\}$.

This result shows that the considerations of \cref{sectionAligned} can -- with a little effort, but without serious difficulties -- be adapted to more general situations. Our assumptions were chosen in such a way that we were able to reuse some results from the aligned case. However, under different assumptions on the parameters, one can proceed similarly, as long as one can choose suitable sectors in the domain and target of the Fourier--Laplace transform keeping the topological situation reasonable.

\begin{ack}
	I would like to thank Marco Hien, Andrea D'Agnolo and Claude Sabbah for inspiring discussions on this subject, as well as for some useful comments during the preparation of this work.
\end{ack}


\begin{thebibliography}{10}
\bibitem{Bjo} J.-E. Björk, \emph{Analytic $\mathcal{D}$-modules and applications}, Mathematics and Its Applications {\bf 247}, Kluwer Academic Publishers, Dordrecht 1993.
\bibitem{Bo12} P. Boalch, \emph{Simply-laced isomonodromy systems}, Publ. Math. Inst. Hautes Études Sci. {\bf 116} (2012), 1--68.
\bibitem{Bo15} P. Boalch, \emph{Global Weyl groups and a new theory of multiplicative quiver varieties}, Geometry and Topology {\bf 19} (2015), 3467--3536.
\bibitem{DHMS} A. D'Agnolo, M. Hien, G. Morando and C. Sabbah, \emph{Topological computation of some Stokes phenomena on the affine line}, Ann. Inst. Fourier {\bf 70} (2020), 739--808
\bibitem{DK16} A. D'Agnolo and M. Kashiwara, \emph{Riemann-Hilbert correspondence for holonomic D-modules}, Publ. Math. Inst. Hautes Études Sci. {\bf 123} (2016), 69--197.
\bibitem{DK18} A. D'Agnolo and M. Kashiwara, \emph{A microlocal approach to the enhanced Fourier--Sato transform in dimension one}, Adv. Math {\bf 339} (2018), 1--59.
\bibitem{DK19} A. D'Agnolo and M. Kashiwara, \emph{Enhanced perversities}, J. reine angew. Math. {\bf 751} (2019), 185--241.
\bibitem{Del} P. Deligne, \emph{Lettre à B. Malgrange du 19/4/1978}, in: Singularités irrégulières, Correspondance et documents, Documents mathématiques {\bf 5}, Société Mathématique de France, Paris (2007), 25--26.
\bibitem{GM03} S. I. Gelfand and Y. I. Manin, \emph{Methods of Homological Algebra} (2nd ed.), Springer Monographs in Mathematics, Springer, Berlin 2003.
\bibitem{HS} M. Hien and C. Sabbah, \emph{The local Laplace transform of an elementary irregular meromorphic connection}, Rend. Sem. Mat. Univ. Padova {\bf 134} (2015), 133--196.
\bibitem{Hohl} A. Hohl, \emph{D-Modules of Pure Gaussian Type from the Viewpoint of Enhanced Ind-Sheaves}, Universität Augsburg, https://opus.bibliothek.uni-augsburg.de/opus4/79690 (2020). Accessed 12 October 2020.
\bibitem{HTT} R. Hotta, K. Takeuchi and T. Tanisaki, \emph{D-Modules, Perverse Sheaves, and Representation Theory}, Progress in Mathematics {\bf 292}, Birkhäuser, Boston 2008.
\bibitem{IT18} Y. Ito and K. Takeuchi. \emph{On Irregularities of Fourier Transforms of Regular Holonomic D-Modules}, Adv. Math. {\bf 366} (2020), 107093.
\bibitem{Kas84} M. Kashiwara, \emph{The Riemann--Hilbert problem for holonomic systems}, Publ. Res. Inst. Math. Sci. {\bf 20} (1984), 319--365.
\bibitem{Kas03} M. Kashiwara, \emph{D-modules and Microlocal Calculus}, Translations of Mathematical Monographs {\bf 217}, Am. Math. Soc., Providence 2003.
\bibitem{Kas16} M. Kashiwara, \emph{Riemann--Hilbert correspondence for irregular holonomic $\mathcal{D}$-modules}, Jpn. J. Math. {\bf 11}, (2016), 113--149.
\bibitem{KS90} M. Kashiwara and P. Schapira, \emph{Sheaves on Manifolds}, Grundlehren der mathematischen Wissenschaften {\bf 292}, Springer, Berlin 1990.
\bibitem{KS01} M. Kashiwara and P. Schapira, \emph{Ind-sheaves}, Astérisque {\bf 271} (2001).
\bibitem{KS06} M. Kashiwara and P. Schapira, \emph{Categories and Sheaves}, Grundlehren der mathematischen Wissenschaften {\bf 332}, Springer, Berlin 2006.
\bibitem{KS16a} M. Kashiwara and P. Schapira, \emph{Regular and irregular holonomic D-modules}, London Mathematical Society Lecture Note Series {\bf 433}, Cambridge University Press, Cambridge 2016.
\bibitem{KS16b} M. Kashiwara and P. Schapira, \emph{Irregular holonomic kernels and Laplace transform}, Sel. Math. New. Ser. {\bf 22} (2016), 55--109.
\bibitem{KL} N. M. Katz and G. Laumon, \emph{Transformation de Fourier et majoration de sommes exponentielles}, Publ. Math. Inst. Hautes Études Sci. {\bf 62} (1985), 145--202.
\bibitem{Mal91} B. Malgrange, \emph{Équations différentielles à coefficients polynomiaux}, Progress in Mathematics {\bf 96}, Birkhäuser, Boston 1991.
\bibitem{Moc10} T. Mochizuki, \emph{Note on the Stokes structure of Fourier transform}, Acta Math. Vietnam. {\bf 35} (2010), 107--158.
\bibitem{Moc14} T. Mochizuki, \emph{Holonomic $\mathcal{D}$-modules with Betti structure}, Mém. Soc. Math. Fr. (N.S.), {\bf 138--139} (2014), Soc. Math. France, Paris, 2014.
\bibitem{Moc16} T. Mochizuki, \emph{Curve test for enhanced ind-sheaves and holonomic D-modules}, \texttt{arXiv:1610.08572v3} (2018).
\bibitem{Moc18} T. Mochizuki, \emph{Stokes shells and Fourier transforms}, \texttt{arXiv:1808.01037v1} (2018).
\bibitem{Sab08} C.\ Sabbah, \emph{An explicit stationary phase formula for the local formal Fourier-Laplace transform}, Contemp. Math. {\bf 474} (2008), 309--330.
\bibitem{Sab13} C. Sabbah, \emph{Introduction to Stokes Structures}, Lecture Notes in Mathematics {\bf 2060}, Springer, Berlin 2013.
\bibitem{Sab16} C. Sabbah, \emph{Differential systems of pure Gaussian type}, Izv. Math. {\bf 80} (2016), 189--220.
\end{thebibliography}
\end{document}